\documentclass[smallextended,numbook,runningheads]{svjour3}     % onecolumn (second format)
\newtheorem{Definition}[theorem]{Definition}
\newtheorem{assumption}[theorem]{Assumption}
\smartqed  % flush right qed marks, e.g. at end of proof
\usepackage{amsmath, amsfonts,amssymb,amscd,marvosym}
 \usepackage[square, sort, numbers]{natbib}
\usepackage{graphicx}
\usepackage{epstopdf} % needed if you have eps figures in a pdflatex manuscript
\usepackage{mathptmx}      % use Times fonts if available on your TeX system
\usepackage{placeins}
\usepackage[showonlyrefs]{mathtools}
\mathtoolsset{showonlyrefs}
\usepackage{hyperref}
\hypersetup{colorlinks=true,citecolor=blue}
\journalname{Noname}

\begin{document}
\title{A Lyapunov and Sacker-Sell spectral stability theory for one-step methods  \thanks{This research was supported in part by NSF grant DMS-1419047.}
%placed here. 
% General acknowledgments should be placed at the end of the article.
}


\author{Andrew J. Steyer         \and
        Erik S. Van Vleck %etc.
}


\institute{Andrew J. Steyer \at
              Sandia National Laboratories; P.O. Box 5800, MS 1320; 
Albuquerque, NM 87185-1320, USA  \\
             % Tel.: +123-45-678910\\
            %  Fax: +123-45-678910\\
              \email{asteyer@sandia.gov}           %  \\
           \and
          Erik S. Van Vleck \at
             Department of Mathematics - University of Kansas; 
1460 Jayhawk Blvd; Lawrence, KS 66045; USA \\
        \email{erikvv@ku.edu}
}

\date{Received: date / Accepted: date}

\maketitle

\begin{abstract}
Approximation theory for Lyapunov and Sacker-Sell spectra based upon QR techniques is used to analyze the stability of a one-step method solving a time-dependent, linear, ordinary differential equation (ODE) initial value problem in terms of the local error.  Integral separation is used to characterize the conditioning of stability spectra calculations.  In an approximate sense the stability of the numerical solution by a one-step method of a time-dependent linear ODE using real-valued, scalar, time-dependent, linear test equations is justified.  This analysis is used to approximate exponential growth/decay rates on finite and infinite time intervals and establish global error bounds for one-step methods approximating uniformly stable trajectories of nonautonomous and nonlinear ODEs.  A time-dependent stiffness indicator and a one-step method that switches between explicit and implicit Runge-Kutta methods based upon time-dependent stiffness are developed based upon the theoretical results.  
\end{abstract}

\section{Introduction}\label{sec:introduction}

Stability plays a central role in determining the time asymptotic behavior of dynamical systems.  In the seminal works of Lyapunov \cite{lyap} and Dahlquist \cite{Dahlquist1956,Dahlquist1959,Dahlquist1963}, stability theories for ordinary differential equation (ODE) initial value problems (IVPs) and methods for their numerical solution were respectively established.  The stability of time-dependent (nonautonomous)  solutions to ODEs can be determined using a variety of techniques, but does not in general reduce to a time-dependent eigenvalue problem (see the third example on page 24 of \cite{Kreiss1978} or the example at the bottom of page 3 of \cite{CoppelBook1978}).  Understanding the stability of numerical methods approximating time-dependent solutions to ODE IVPs is important for preventing spurious computational modes, detecting and quantifying stiffness, and controlling the global error.  The complementary dynamical systems viewpoint is that the dynamics of numerical solutions should mimic the dynamics of differential equations.  In this paper we embrace both of these points of view and use Lyapunov and Sacker-Sell spectral theory to develop a time-dependent stability theory for one-step methods approximating solutions of ODE IVPs.

Our contribution is to establish a Lyapunov stability theory for variable step-size one-step methods approximating time-dependent solutions of ODE IVPs that can fail to satisfy the hypotheses of AN- and B-stability theories (see Equation \ref{eq:2dlin} below for an example of such an ODE).  We use integral separation, the time-dependent analog of gaps between eigenvalues, to characterize the conditioning of computations of time-dependent stability spectra.  The main results, Theorems \ref{thm:maintheorem1} and \ref{thm:maintheorem2}, harness the fact that the numerical solution of a nonautonomous linear ODE of dimension $d$ by a one-step method defines a nonautonomous linear difference equation of dimension $d$.   A time-dependent and orthogonal change of variables is employed to transform to a linear difference equation with an upper triangular coefficient matrix, from which spectral endpoints and integral separation properties can be determined from the diagonal entries.   Theorem \ref{thm:maintheorem1} concludes that if the coefficient matrix of the ODE is bounded and continuous, then the Sacker-Sell spectrum of the numerical solution approximates that of the ODE.  Theorem \ref{thm:maintheorem2} concludes that if the ODE has an integral separation structure, then the Lyapunov and Sacker-Sell spectrum of the numerical solution accurately approximate the spectra of the ODE in terms of the local truncation error.  The endpoints of the spectra can then be estimated from the diagonal entries of the resulting upper triangular coefficient matrix.  These theorems together with Lemma \ref{lem:localgrowthcts} are then used to justify characterizing the stability of a one-step method solving a nonautonomous linear ODE of dimension $d$ with $d$ scalar, real-valued, nonautonomous linear test equations.  We then show (Theorem \ref{thm:1dcounterexample}) the necessity of controlling time-dependent stability through a step-size restriction and prove Theorem \ref{thm:stabfcntheorem} showing that the stability of a Runge-Kutta method solving a complex-valued, scalar, nonautonomous linear test equation can be approximately characterized by when the time-averages of the coefficient function lie in the linear stability region of the method.

The linear stability results are applied to prove two theorems (Theorems \ref{thm:nonlin1} and \ref{thm:nonlin2}) on the numerical solution by a one-step method of a uniformly exponentially stable solution of a nonlinear and nonautonomous ODE.  Theorem \ref{thm:nonlin1} provides a global error bound for the numerical solution of a uniformly exponentially stable trajectory of a nonlinear IVP in terms of the local truncation error.  Theorem \ref{thm:nonlin2} shows that the numerical approximation of a uniformly exponentially stable trajectory is uniformly exponentially stable with decay rates approximately those of the exact solution.  The nonlinear results, which draw on the spirit of the one-step approximation theory developed in  \cite{Beyn1987}, \cite{Eirola1988}, and \cite{KL1986}, show that the spectral stability of the numerical solution of a nonlinear ODE IVP by a one-step method can be characterized and quantified in terms of the spectral stability of the numerical solution of the associated linear variational equation. 

The linear and nonlinear theoretical results are applied in Section \ref{sec:numexp}.  In Section \ref{sec:stiffdetect} we develop an efficient time-dependent stiffness indicator and in Section \ref{sec:qrimex} we develop a one-step method, referred to as a QR-IMEX-RK method, that switches between using implicit and explicit Runge-Kutta methods. Our stiffness indicator is computed using Steklov averages approximated from the discrete QR method for computing Lyapunov exponents \cite{DVV4}.  This indicator is in general more efficient to compute than methods such as that proposed in Definition 4.1 of \cite{GSC2015} that require approximating logarithmic norms or time-dependent eigenvalues and additionally our indicator is able to detect stiffness in IVPs with non-normal Jacobians where logarithmic norms and time-dependent eigenvalues can fail to indicate stiffness.  Being able to detect stiffness efficiently and robustly is necessary in the context of our QR-IMEX-RK methods where we switch between using an implicit or explicit Runge-Kutta method based on where approximate Steklov averages are at each time-step in relation to the linear stability regions of the explicit and implicit methods. 
 
The stability of numerical solutions of ODE IVPs is a classic topic in numerical analysis dating back at least to the PhD thesis of Dahlquist (published as \cite{Dahlquist1959}) and also \cite{Dahlquist1956,Dahlquist1963} where concepts such as A-stability were first introduced.  Other stability theories for the numerical solution of nonautonomous and nonlinear ODE IVPs, such as B-stability \cite{Butcher1975} or algebraic stability and AN-stability \cite{BurrageButcher1979} provide an analysis for various classes of ODEs that are monotonically contracting.  The equivalences amongst these nonlinear and nonautonomous stability theories are investigated in \cite{Butcher1987}.  In the case of Runge-Kutta methods the analysis in AN-, B-, and algebraic stability requires that the methods be implicit and at least A-stable while our analysis holds so long as the method is convergent.

The theory developed in this work is based on the time-dependent spectral stability theories of the Lyapunov and Sacker-Sell spectra.  We refer to the monograph \cite{ADR} by Adrianova as a general reference on time-dependent stability and related topics such as integral separation.  The theory of Lyapunov exponents and the associated Lyapunov spectrum arose from the thesis of Lyapunov \cite{lyap}.  The Sacker-Sell spectrum, defined by the values for which a shifted time-dependent linear ODE does or does not admit exponential dichotomy, first appears in the literature in the the fundamental 1978 paper \cite{SackerSell1978} of Sacker and Sell.  The Lyapunov spectrum characterizes the exponential stability while the Sacker-Sell spectrum characterizes the uniform exponential stability of a nonautonomous linear ODE or difference equation. 

In this paper we apply the QR approximation theory for Lyapunov and Sacker-Sell spectra (see e.g. \cite{DVV00,DVV1,DVV3,DVV7,DRVV}, \cite{CDVV2}, \cite{VV1}, and \cite{BVV}).  QR approximation theory constructs the orthogonal factor in a QR factorization of a fundamental matrix solution (in continuous or discrete time) to transform a linear system to one with an upper triangular coefficient matrix.  Then, assuming either that the system has an integral separation structure or a bounded and continuous coefficient matrix, the endpoints of the Lyapunov or Sacker-Sell spectrum respectively can be approximated from the diagonal entries of the transformed upper triangular matrix.

The development of our theory is motivated by the following ODE:
\begin{equation}\label{eq:2dlin}
\dot{x} = A(t)x, \quad A(t) = L(t)C(t)L(t)^T, \quad t > 0
\end{equation}
$$C(t) = \left[\begin{array}{cc}\lambda_1 & \beta(t) \\ 0 & \lambda_2\end{array}\right], \quad L(t) = \left[\begin{array}{cc}\cos(\omega(t)) & -\sin(\omega(t)) \\ \sin(\omega(t)) & \cos(\omega(t))\end{array} \right], $$
where $\lambda_1 > 0 > \lambda_2$ with $\lambda_1+\lambda_2 < 0$, $\beta(t) = \beta_0(1+ \cos(a_1 t)/(1+\beta_1 t^2))$, $\omega(t) = a_2 t$, $(\lambda_1+\lambda_2)^2-4(a_1(\beta_0+a_1)+\lambda_1 \lambda_2)>0$, and $a_1,a_2,\beta_0,\beta_1 \geq 0$.  The ODE \eqref{eq:2dlin} does not satisfy the hypotheses of B-stability theory since there exists $v,w\in \mathbb{R}^{2}$ so that $(v-w)^T A(0)^T (v-w) > 0$ nor AN-stability since $\lambda_1 > 0$ is one of the time-dependent eigenvalues of $A(t)$.  However, by using the change of variables $x = L(t)y$ and Theorem 4.3.2 of \cite{ADR}, it follows that zero is an asymptotically stable equilibrium of \eqref{eq:2dlin}.

If we solve \eqref{eq:2dlin} using the implicit Euler method with step-size $h_0>0$ and initial condition $(0,0)^T \neq x_0 \in \mathbb{R}^2$, then the numerical solution $\{x_n\}_{n=0}^{\infty}$ satisfies the following linear difference equation
\begin{equation}\label{eq:impeuler}
x_{n+1} = [I-h_0 A(t_{n+1})]^{-1} x_n, \quad n \geq 0.
\end{equation}
The implicit Euler method is AN- and L-stable and the expectation would be that it should produce a decaying solution to \eqref{eq:2dlin} with no step-size restriction.  However, if $a_1=a_2=2\pi$, $h_0 = 1$, and $\lambda_1  \in (0,1)$, then the solution of \eqref{eq:impeuler} with $x_0 \neq (0,0)^T$ is such that $\|x_n\| \rightarrow \infty$ as $n \rightarrow \infty$ at a rate of $(1-\lambda_1)^n$ where $\|\cdot\|$ is any norm on $\mathbb{R}^2$.  In Section \ref{sec:linearresults} we prove that there is an $h^* > 0$ so that if $h_0 \in (0,h^*$), then solutions of \eqref{eq:impeuler} decay to zero.

The rest of this paper is organized as follows. In Section \ref{sec:preliminaries} we introduce some definitions, notation, and necessary background material.  In Section \ref{sec:linearresults} we state Theorems \ref{thm:maintheorem1} and \ref{thm:maintheorem2} which are subsequently proved in Section \ref{sec:linearresultsproofs}.  We prove Theorems \ref{thm:1dcounterexample} and \ref{thm:stabfcntheorem} in Section \ref{sec:1dresults} which is dedicated to the thorough analysis of a scalar, nonautonomous linear test equation. The nonlinear stability results, Theorems \ref{thm:nonlin1} and \ref{thm:nonlin2}, are stated and proved in Section \ref{sec:nonlin}.  In Section \ref{sec:numexp} we develop a time-dependent stiffness indicator and an algorithm for switching between implicit and explicit Runge-Kutta methods based on time-dependent stiffness which are tested on several interesting and challenging problems.  Concluding remarks are given in Section \ref{sec:afterword} and Butcher tableaux with stability and accuracy properties for methods used in Section \ref{sec:numexp} are given in the appendix in Section \ref{sec:appendix}.

%===================================================================================================================================
%===================================================================================================================================
%===================================================================================================================================
%===================================================================================================================================
%===================================================================================================================================

\section{Preliminaries}\label{sec:preliminaries}

\subsection{Stability of initial value problems}

Consider the nonautonomous and nonlinear ODE
 \begin{equation}\label{eq:ode}
 \dot{x} = f(x,t)
 \end{equation}
where $f:\mathbb{R}^{d}\times (\tau_0,\infty) \rightarrow \mathbb{R}^{d}$ for some positive integer $d$ and $\tau_0 \geq -\infty$.  We assume that $f(x,\cdot)$ is bounded for each fixed $x \in \mathbb{R}^d$ and that $f$ is at least $C^1(\mathbb{R}^{d}\times (\tau_0,\infty))$ and sufficiently smooth so that each IVP
\begin{equation}\label{eq:odeivp}
\left\{
\begin{array}{lcr}
\dot{x} = f(x,t)\\
x(t_0) = x_0
\end{array}
\right.
\end{equation}
has a unique and globally defined solution $x(t;x_0,t_0)$ for all initial conditions $x_0 \in \mathbb{R}^{d}$ and initial times $t_0 > \tau_0$.

Fix an arbitrary norm $\|\cdot\|$ on $\mathbb{R}^d$ and use the same symbol $\|\cdot\|$ to denote the induced matrix norm on $\mathbb{R}^{d\times d}$.  Henceforth, whenever we use the word stability we are referring to Lyapunov stability in either continuous or discrete time.  Assume that the solution $x(t;x_0,t_0)$ is bounded in $t$ and consider the linear variational equation:

\begin{equation}\label{eq:linvareq}
\dot{x} = A(t)x, \quad t > t_0, \quad A(t)=Df(x(t;x_0,t_0),t), \quad D:= \partial/\partial x.
\end{equation}
Since $x(t;x_0,t_0)$ is bounded in $t$ and $f$ is $C^1$ it follows that $A(\cdot)$ is bounded and continuous.
\begin{definition}
We say that \eqref{eq:linvareq} is exponentially stable if for any fundamental matrix solution $X(t)$ of \eqref{eq:linvareq} there exists $\gamma > 0$ and $K > 0$ so that 
\begin{equation}
\|X(t)\| \leq K e^{-\gamma (t-t_0)}\|X(t_0)\|, \quad t\geq t_0.
\end{equation}
\eqref{eq:linvareq} is said to be uniformly exponentially stable if for any fundamental matrix solution $X(t)$ of \eqref{eq:linvareq} there exists $\gamma > 0$ and $K > 0$ so that 
$$\|X(t)\| \leq K e^{-\gamma (t-s)}\|X(s)\|, \quad t\geq s \geq t_0.$$
\end{definition}
We characterize exponential and uniform exponential stability using Lyapunov and Sacker-Sell spectra (see \cite{DVV3} for a review of the definitions of these spectra).  If the Lyapunov spectrum of \eqref{eq:linvareq} is contained in $(-\infty,0)$, then \eqref{eq:linvareq} is exponentially stable.  A sufficient condition for uniform exponential stability of zero is that the Sacker-Sell spectrum of \eqref{eq:linvareq} is contained in $(-\infty,0)$.  The linear concepts of exponential stability have the following analogous definitions in the nonlinear setting.
\begin{definition}
The trajectory $x(t;x_0,t_0)$ is exponentially stable if there exists $\gamma,K, \delta > 0$ so that if $\|u_0-x_0\| < \delta$, then $\|x(t;u_0,t_0)-x(t;x_0,t_0)\| \leq K e^{-\gamma(t-t_0)}\|u_0-x_0\|$ for all $t \geq t_0$.  We say that $x(t;x_0,t_0)$ is uniformly exponentially stable if there exists $\gamma,K, \delta > 0$ so that for any $t \geq s \geq t_0$, if $\|u_s-x(s;x_0,t_0)\| < \delta$, then $\|x(t;u_s,s)-x(t;x_0,t_0)\| \leq K e^{-\gamma(t-s)}\|u_s-x(s;x_0,t_0)\|$.
\end{definition}
If the linear variational equation \eqref{eq:linvareq} of $x(t;x_0,t_0)$ is uniformly exponentially stable and $f$ is sufficiently smooth, then $x(t;x_0,t_0)$ is a uniformly exponentially stable trajectory of \eqref{eq:ode}.  However, if the linear variational equation of $x(t;x_0,t_0)$ is exponentially stable, but not uniformly exponentially stable, then we cannot even guarantee that $x(t;x_0,t_0)$ is stable (see \cite{Perron1930} or Equation 14 in \cite{LK2007} for an example).

\subsection{One-step methods}

A one-step method is an approximation to solutions of ODE IVPs \eqref{eq:odeivp} of the form 
\begin{equation}\label{eq:onestep}
x_{n+1} = \varphi(x_n,t_n; f,h)
\end{equation}
where $x_n \approx x(t_n;x_0,t_0)$, $f=f(x,t)$ is the right-hand side function of \eqref{eq:ode}, $h$ is a sequence of step-sizes $h = \{h_n\}_{n=0}^{\infty}$ which we always assume is such that $0 < \text{inf}_{n \geq 0}h_n \leq \text{sup}_{n \geq 0}h_n < \infty$, and  $t_{n+1} = t_n + h_n$ for all $n \geq 0$.  We let $\|\cdot\|_{\infty}$ denote the $l^{\infty}$ norm for sequences with $\|h\|_{\infty} = \text{sup}_{n \geq 0}h_n$.  We say that the one-step method \eqref{eq:onestep} has local truncation error of order $p \in \mathbb{N}$ if there exists $h^* > 0$ so that if $f \in \mathbb{C}^{p+1}$ and $\|h\|_{\infty} \in (0,h^*)$, then the Taylor expansion of any solution $x:(t_0,\infty)\rightarrow \mathbb{R}^d$ of \eqref{eq:ode} is:
$$x(t_{n+1}) -\varphi(x(t_n),t_n; f(x(t_n),t_n),h) = K_n h_n^{p+1}, \quad n \geq 0.$$
where $K_n= K(t_n)$ defines some sequence depending on $x(t)$ and its derivatives (in particular $\{K_n\}_{n=0}^{\infty}$ will be bounded when $x(t)$ is bounded in $t$).  The numerical solution of a linear ODE of the form \eqref{eq:linvareq} using a sequence of step-sizes $h=\{h_{n}\}_{n=0}^{\infty}$ by a one-step method with local truncation error of order $p \geq 1$ is a nonautonomous linear difference equation of the form $x_{n+1} = \Phi^A(n;h)x_n$.  This fact is exploited throughout the remainder of the paper.

\subsection{Spectral theory for continuous time systems}

Consider the following $d$ dimensional nonautonomous linear ODE
\begin{equation}\label{eq:lineq}
\dot{x} = A(t)x, \quad t > t_0
\end{equation}
where $A:(t_0,\infty)\rightarrow \mathbb{R}^{d \times d}$ is bounded and continuous.  The continuous QR method for transforming \eqref{eq:lineq} to upper triangular form is as follows.  Consider the following ODE \cite{DRVV}:
\begin{equation}\label{eq:Qeqn}
\dot{Q}(t) = Q(t)S(Q(t),A(t)), \quad S(Q,A)_{ij} = \left\{
\begin{array}{cc}
(Q^T AQ)_{i,j}, & \quad i>j \\
0, & \quad i=j \\
-(Q^T A Q)_{i,j}, & \quad i<j 
\end{array}
\right. .
\end{equation}

Each orthogonal matrix solution $Q(t) \in \mathbb{R}^{d\times d}$ of \eqref{eq:Qeqn} defines a linear system
\begin{equation}\label{eq:lineqB0}
\dot{y} = B(t)y, \quad Q^T(t)A(t)Q(t)-Q^T(t)\dot{Q}(t), \quad t > t_0
\end{equation}
where $B(t)$ is upper triangular.  We refer to \eqref{eq:lineqB0} as a corresponding upper triangular system (or ODE) to \eqref{eq:lineq}. Since $x = Q(t)y$ is a Lyapunov transformation the Lyapunov and Sacker-Sell spectral intervals of \eqref{eq:lineq} coincide with those of any corresponding upper triangular system.

\begin{theorem}[Theorems 2.8, 5.5, and 6.1 of \cite{DVV3}]\label{thm:bddctsthm}
Let $B:(t_0,\infty) \rightarrow \mathbb{R}^{d\times d}$ be bounded, continuous, and upper triangular and let $\Sigma_{ED} = \cup_{i=1}^{d}[\alpha_i,\beta_i]$ denote the Sacker-Sell spectrum of the ODE $\dot{y} = B(t)y$.  For $i=1,\hdots,d$ we have:
\begin{equation}\label{eq:SSformula}
\alpha_i = \liminf_{0 < H\rightarrow \infty}\left(\inf_{t \geq t_0}\dfrac{1}{H}\int_{t}^{t+H}B_{i,i}(\tau)d\tau\right), \quad \beta_i = \limsup_{0 < H\rightarrow \infty}\left(\sup_{t \geq t_0}\dfrac{1}{H}\int_{t}^{t+H}B_{i,i}(\tau)d\tau\right).
\end{equation}\qed

\end{theorem}
For a bounded and continuous $A(\cdot)$, the Sacker-Sell spectrum of \eqref{eq:lineq} is continuous with respect to $L^{\infty}(t_0,\infty)$ perturbations of $A(t)$ (for a proof see Theorem 6 of \cite{SackerSell1978} or Chapter 4 of \cite{CoppelBook1978}).  For the Lyapunov spectrum to be continuous an additional hypothesis must be placed on \eqref{eq:lineq}.

\begin{Definition}\label{def:iss}
Suppose that $B:(t_0,\infty) \rightarrow \mathbb{R}^{d\times d}$ is bounded, continuous, and upper triangular and that for any $i <j$ one of the two following conditions hold:
\begin{enumerate}
\item $B_{i,i}$ and $B_{j,j}$ are integrally separated: there exists $a_{i,j} > 0$ and $b_{i,,j} \in \mathbb{R}$ so that if $t \geq s \geq t_0$, then
\begin{equation}\label{eq:IS}
\int_{s}^{t}B_{i,i}(\tau) - B_{j,j}(\tau)d\tau \geq a_{i,j}(t-s)+b_{i,j}.
\end{equation}
\item For every $\varepsilon > 0$ there exists $M_{i,j}(\varepsilon) > 0$ so that if $t \geq s \geq t_0$, then
\begin{equation}\label{eq:nonIS}
\left|\int_{s}^{t}B_{i,i}(\tau)-B_{j,j}(\tau) d\tau\right| \leq M_{i,j} + \varepsilon(t-s).
\end{equation}
\end{enumerate}
Then we say that $\dot{y}=B(t)y$ and $B(t)$ have an integral separation structure.  If the first condition is satisfied for all $i < j$, then we say that $B(t)$ and $\dot{y}=B(t)y$ are integrally separated.   If the system \eqref{eq:lineq} has a corresponding upper triangular system that has an integral separation structure, then we say that \eqref{eq:lineq} has an integral separation structure and if the corresponding upper triangular system is integrally separated, then we say that \eqref{eq:lineq} is integrally separated.
\end{Definition}
 Integral separation is a generic property (see page 21 of \cite{Palmer}) for linear equations on the half-line with the sup-norm topology.   This, together with the following theorem, show why it is natural to assume that a linear equation \eqref{eq:lineq} has an integral separation structure when approximating Lyapunov spectral intervals.
\begin{theorem}[Theorem 5.1 in \cite{DVV3}]\label{thm:suffconstablyap}
Assume that $B:(t_0,\infty) \rightarrow \mathbb{R}^{d\times d}$ has an integral separation structure and let $\Sigma_L = \cup_{i=1}^{d}[\eta_i,\mu_i]$ denote the Lyapunov spectrum of the ODE $\dot{y} = B(t)y$.  Then the Lyapunov spectrum of $\dot{y}=B(t)y$ is continuous with respect to $L^{\infty}(t_0,\infty)$ perturbations of $B(t)$ and for $i=1,\hdots,d$ we have:
\begin{equation}\label{suffconstablyap.2}
\eta_i =\liminf_{0 < t\rightarrow \infty}\dfrac{1}{t}\int_{t_0}^{t_0+t}B_{i,i}(\tau)d\tau, \quad \mu_i =\limsup_{0 < t\rightarrow \infty}\dfrac{1}{t}\int_{t_0}^{t_0+t}B_{i,i}(\tau)d\tau.
\end{equation}\qed
\end{theorem}
We remark that if \eqref{eq:lineq} has a corresponding upper triangular system that has an integral separation structure (respectively is integrally separated), then every corresponding upper triangular system has an integral separation structure (respectively is integrally separated).  If the system \eqref{eq:lineq} does not have an integral separation, then the Lyapunov spectrum may be unstable (see Example 5.4.2 of \cite{ADR}).  Theorems \ref{thm:bddctsthm} and \ref{thm:suffconstablyap} are the basis for the assumptions that we place on \eqref{eq:lineq} in Section \ref{sec:mainresults}.

\subsection{Spectral theory for discrete time systems}

Consider a family of nonautonomous linear difference equations of the form
\begin{equation}\label{eq:dislineq}
x_{n+1} = \Phi^A(n;h)x_n, \quad n \geq 0
\end{equation}
where $x_n \in \mathbb{R}^{d}$, $h=\{h_n\}_{n=0}^{\infty}$ is a sequence of step-sizes, and each $\{\Phi^A(n;h)\}_{n=0}^{\infty}$ is a bounded matrix sequence where each $\Phi^A(n;h) \in \mathbb{R}^{d\times d}$ is invertible.

We construct an orthogonal change of variables transforming \eqref{eq:dislineq} to upper triangular form as follows.  Let $Q_0 \in \mathbb{R}^{d\times d}$ be an orthogonal matrix and fix some step-size sequence $h$.  Since $\Phi^A(n;h)$ is invertible for all $n \geq 0$ we can form unique QR factorizations $\Phi^A(n;h) Q_n = Q_{n+1} R^A(n;h)$ where $Q_{n+1} \in \mathbb{R}^{d\times d}$ is orthogonal and $R^A(n;h) \in \mathbb{R}^{d\times d}$ is upper triangular with positive diagonal entries. This process is referred to as a discrete QR iteration.  The system $u_{n+1} = R^A(n;h)u_n$ where $R^A(n;h) = Q_{n+1}^T \Phi^A(n;h)Q_n$ is referred to as a corresponding upper triangular system and its Lyapunov and Sacker-Sell spectra coincide with those of \eqref{eq:dislineq}.

\begin{theorem}[Section 5.1 of \cite{BredaVanVleck2014} or Corollary 3.25 of \cite{Potzsche2012}]\label{thm:ssdiss}
Fix some step-size sequence $h$ and assume that the sequence $\{R^A(n;h)\}_{n=0}^{\infty}$ is bounded and that each $R^A(n;h)$ is invertible and upper triangular.  Let $\Sigma_{ED}^{A} = \cup_{i=1}^{d}[\alpha_i^A,\beta_i^A]$ denote the Sacker-Sell spectrum of $u_{n+1} = R^A(n;h)u_n$.  Then for $i=1,\hdots,d$ we have
\begin{equation}
\small
\alpha_i^A = \liminf_{1 \leq m \rightarrow \infty}\left(\inf_{n \geq 0}\dfrac{1}{t_{n+m}-t_n}\ln|\prod_{k=n+m}^{n}R_{i,i}^{A}(k;h)|\right), \quad \beta_i^A = \limsup_{1 \leq m \rightarrow \infty}\left(\sup_{n \geq 0}\dfrac{1}{t_{n+m}-t_n}\ln|\prod_{k=n+m}^{n}R_{i,i}^{A}(k;h)|\right).\end{equation} \qed
\end{theorem}
Theorem 4.1 of \cite{Potzsche2013} implies that the Sacker-Sell spectrum of \eqref{eq:dislineq} is continuous with respect to $l^{\infty}(\mathbb{N})$ perturbations of the coefficient matrix.  Discrete integral separation characterizes when the Lyapunov spectrum of \eqref{eq:dislineq} is continuous.

\begin{Definition}\label{adiss}
Consider $u_{n+1} = R^A(n;h) u_n$ where each $R^A(n;h)\in \mathbb{R}^{d \times d}$ is invertible and upper triangular, the sequence $\{R^A(n;h)\}_{n=0}^{\infty}$ is bounded, and $\text{inf}_{n \geq 0}R_{i,i}^A(n;h) > 0$ for $i=1,\hdots,d$.  Let $p \geq 1$ and suppose there exists an $h^* > 0$ so that if $\|h\|_{\infty} \in (0, h^*)$ and $i <j$, then one of the two following conditions hold:
\begin{enumerate}
\item  $R^A_{i,i}(n;h)$ and $R^A_{j,j}(n;h)$ are discretely integrally separated: there exists $b_{i,j} \in \mathbb{R}$ and $a_{i,j} > 0$ so that if $n \geq m$, then 
\begin{equation}\label{eq:adiss1}
\prod_{k=m}^{n} R_{i,i}^{A}(k;h) (R_{j,j}^A(k;h))^{-1} \geq \exp\left(a_{i,j}(t_n-t_m)+b_{i,j}\right).
\end{equation}
\item $R_{i,i}^A(n;h)$ and $R_{j,j}^A(n;h)$ satisfy that there exists $K_{i,j} > 0$ such that for each $\varepsilon > 0$ there exist $M_{i,j}> 0$ so that if $n \geq m$, then
\begin{equation}\label{eq:adiss2}
\left|\ln\left(\prod_{k=m}^{n}R_{i,i}^A(k;h) (R_{j,j}^A(k;h))^{-1}\right)\right| \leq M_{i,j} + (\varepsilon + K_{i,j}\|h\|_{\infty}^{p})(t_n-t_m).
\end{equation}
\end{enumerate}
We refer to such a system as a system with a p-approximate discrete integral separation structure and say that $R^A(n;h)$ has a p-approximate discrete integral separation structure.  If the first condition is satisfied for all $i < j$, then we say that $R^A(n;h)$ is discretely integrally separated.  If \eqref{eq:dislineq} has a corresponding upper triangular system with a p-approximate discrete integral separation structure, then we say that \eqref{eq:dislineq} has a p-approximate discrete integral separation structure.
\end{Definition}
 The following theorem follows from the results proved in \cite{BVV,DVV3,DVV4}.
\begin{theorem}\label{thm:lediscrete}
Suppose $u_{n+1} = R^A(n;h) u_n$ is a system with a p-approximate discrete integral separation structure with Lyapunov spectrum $\Sigma_L^A = \cup_{i=1}^{d}[\eta_i^A,\mu_i^A]$.  Then, there exists $h^* > 0$ so that if $\|h\|_{\infty} \in (0,h^*)$, then for $i=1,\hdots,d$ we have:
$$\eta_i^A = \liminf_{n \rightarrow \infty} s_i^A(n) + E_i(n;h), \quad \mu_i^A = \limsup_{n \rightarrow \infty}s_i^A(n) + F_i(n;h)$$
where $\|E_i(n;h)\|,\|F_i(n;h)\|= \mathcal{O}(\|h\|_{\infty}^{p})$ and $s_i^A(n) = \frac{\sum_{k=0}^{n}\ln(R_{i,i}^A(k;h))}{t_n-t_0}$.  If $R^A(n;h)$ is discretely integrally separated, then $E_i(n;h) \equiv F_i(n;h) \equiv 0$ for $i=1,\hdots,d$.\qed
\end{theorem}
 
Consider the perturbed system $z_{n+1} = (\Phi^A(n;h) + F_n ) z_n$ and assume that $\Phi^A(n;h)$ and $\Phi^A(n;h) + F_n$ are bounded and invertible for all $n \geq 0$.  Fix $Q_0 = \overline{Q}_0$ orthogonal and inductively construct unique QR factorizations $\Phi^A(n;h) Q_n = Q_{n+1} R^A(n;h)$ and $(\Phi^A(n;h) + F_n) \overline{Q}_n = \overline{Q}_{n+1} \overline{R}^A(n;h)$ where $Q_n$ and $\overline{Q}_n$ are orthogonal and $R^A(n;h)$ and $\overline{R}^A(n;h)$ are upper triangular with positive diagonal entries.

\begin{theorem}[Theorem 7.7 in \cite{BVV} and Theorem 4.1 in \cite{VV1}]\label{thm:disQR}

Suppose that $\overline{R}^A(n;h)$ has a p-approximate discrete integral separation structure.  If $G:= \text{sup}_{n \geq 0}\|G_n\| = \text{sup}_{n \geq 0}\text{max}\{\|E_n\|,\|F_n\|\}$ where $E_n = \overline{Q}_{n+1}^T E_n \overline{Q}_n$ is sufficiently small, then there exists an $h^* > 0$ so that if $h$ is any sequence of step-sizes with $\|h\|_{\infty} \in (0, h^*)$, then there exists an orthogonal sequence of matrices $\{\tilde{Q}_n\}_{n=0}^{\infty}$ and $K > 0$ such that 
$$\tilde{Q}_{n+1} R^A(n;h) = [\overline{R}^A(n;h)+E_n]\tilde{Q}_n, \quad \|\tilde{Q}_n-I\| \leq KG, \quad n \geq 0.$$\qed
\end{theorem}

%=================================================================================================================================
%=================================================================================================================================
%=================================================================================================================================
%=================================================================================================================================
%=================================================================================================================================

\section{Main Results}\label{sec:mainresults}

\subsection{Statement of the main results for linear ODEs}\label{sec:linearresults}

We henceforth fix a one-step method $\mathcal{M}$ with local truncation error of order $p \geq 1$ and consider a linear system \eqref{eq:lineq} with Sacker-Sell spectrum $\Sigma_{ED}=\cup_{i=1}^{d}[\alpha_i,\beta_i]$ and Lyapunov spectrum $\Sigma_L = \cup_{i=1}^{d}[\eta_i,\mu_i]$.  We make use of the following assumptions to characterize the approximation properties of these two spectra.

\begin{assumption}\label{as:assumption1}
The integer $p \geq 1$ and the coefficient matrix $A(t)$ of \eqref{eq:lineq} is bounded and at least $C^{p+1}$.
\end{assumption}

\begin{assumption}\label{as:assumption2}
The ODE \eqref{eq:lineq} satisfies Assumption \ref{as:assumption1} and in addition there is a corresponding upper triangular ODE
\begin{equation}\label{eq:lineqB}
\dot{y}(t) = B(t)y(t)
\end{equation}
with $x = Q(t)y$ that has an integral separation structure defined by the estimates in Definition \ref{def:iss}.  

\end{assumption}

Let $x_{n+1} = \Phi^A(n;h) x_n$ denote the numerical solution of \eqref{eq:lineq} with initial condition $x(t_0)=x_0$ using $\mathcal{M}$ with some sequence of step-sizes $h = \{h_n\}_{n=0}^{\infty}$ and let $y_{n+1} = \Phi^B(n;h) y_n$ denote the numerical solution of \eqref{eq:lineqB} using $\mathcal{M}$ with the same sequence of step-sizes and initial condition $y(t_0)=y_0 :=Q_0^T x_0$.  We shall always assume that $\|h\|_{\infty}$ is so small that $\Phi^A(n;h)$ and $\Phi^B(n;h)$ are both uniformly bounded and invertible for all $n \geq 0$.  The matrices $\Phi^B(n;h)$ are upper triangular since $B(t)$ is upper triangular and each diagonal entry $\Phi^B_{j,j}(n;h)$ is such that $y_{n+1}^j = \Phi_{j,j}^B(n;h)y_n^j$ is the numerical solution of the scalar equation $\dot{y}_j(t) = B_{j,j}(t)y_j(t)$ with $y_j(t_0) = 1$ using $\mathcal{M}$ and the sequence of step-sizes $h$.  Since $\Phi^A(n;h)$ is invertible for $n \geq 0$ we can inductively construct unique QR factorizations of $\Phi^A(n;h)Q_n$ as $\Phi^A(n;h) Q_n =  Q_{n+1}R^A(n;h)$ where each $Q_{n}$ is orthogonal, $Q_0 = Q(t_0)$, and $R^A(n;h)$ is upper triangular with positive diagonal entries. 

For the remainder of Section \ref{sec:mainresults} we denote the Lyapunov and Sacker-Sell spectra of $x_{n+1} = \Phi^A(n;h)x_n$ by $\Sigma_{ED}^A=\sum_{i=1}^d [\alpha_i^A,\beta_i^A]$ and $\Sigma_L^A=\cup_{i=1}^d [\eta_i^A,\mu_i^A]$ and those of $y_{n+1}=\Phi^B(n;h)y_n$ by $\Sigma_{ED}^B=\sum_{i=1}^d [\alpha_i^B,\beta_i^B]$ and $\Sigma_L^B=\cup_{i=1}^d [\eta_i^B,\mu_i^B]$.  We do not explicitly express the dependence of the spectra of these discrete systems on $h$. The following two theorems are proved in Section \ref{sec:linearresultsproofs}.

\begin{theorem}\label{thm:maintheorem1}
Suppose \eqref{eq:lineq} satisfies Assumption \ref{as:assumption1}.  Given $\varepsilon > 0$, there exists $h^*> 0$ so that if $h$ is any sequence of step-sizes with $\|h\|_{\infty} \in (0, h^*)$, then for $i=1,\hdots,d$:
$$|\alpha_i^A-\alpha_i| < \varepsilon, \quad |\beta_i^A-\beta_i| < \varepsilon, \quad \alpha_i^B = \alpha_i+\mathcal{O}(\|h\|_{\infty}^p), \quad \beta_i^B = \beta_i + \mathcal{O}(\|h\|_{\infty}^p).$$\qed
\end{theorem}

\begin{theorem}\label{thm:maintheorem2}
Suppose \eqref{eq:lineq} satisfies Assumption \ref{as:assumption2}.  There exists $h^* > 0$ so that if $h$ is any sequence of step-sizes with $\|h\|_{\infty}  \in (0, h^*)$, then the following three conclusions hold.
 \begin{enumerate}
 \item The systems $y_{n+1} = \Phi^B(n;h)y_n$ and $u_{n+1} = R^A(n;h)u_n$ have p-approximate discrete integral separation structures and $\|R^A(n;h)-\Phi^B(n;h)\| = \mathcal{O}(\|h\|_{\infty}^{p+1})$.

\item For $i=1,\hdots,d$ we have $\alpha_i^A = \alpha_i^B+\mathcal{O}(\|h\|_{\infty}^p)= \alpha_i + \mathcal{O}(\|h\|_{\infty}^p)$ and $\beta_i^A = \beta_i^B + \mathcal{O}(\|h\|_{\infty}^p) = \beta_i + \mathcal{O}(\|h\|_{\infty}^p)$.

\item For $i=1,\hdots,d$ if $s_i^A(n) := \frac{\sum_{k=0}^{n}\ln (R_{i,i}^A(k;h))}{t_n-t_0}$, $s_i^B(n) := \frac{\sum_{k=0}^{n}(\ln\Phi_{i,i}^B(k;h))}{t_n-t_0}$, then  
$$\eta_i^A =\liminf_{n\rightarrow \infty}s_i^A(n) + \mathcal{O}(\|h\|_{\infty}^p), \quad  \mu_i^A=\limsup_{n\rightarrow \infty} s_i^A(n) + \mathcal{O}(\|h\|_{\infty}^p),$$
 $$\eta_i^B = \liminf_{n\rightarrow \infty} s_i^B(n) + \mathcal{O}(\|h\|_{\infty}^p), \quad \mu_i^B = \limsup_{n\rightarrow \infty} s_i^B(n) + \mathcal{O}(\|h\|_{\infty}^p),$$ 
 $$\eta_i^A = \eta_i^B +\mathcal{O}(\|h\|_{\infty}^p) = \eta_i + \mathcal{O}(\|h\|_{\infty}^p), \quad \mu_i^A = \mu_i^B +\mathcal{O}(\|h\|_{\infty}^p) = \mu_i + \mathcal{O}(\|h\|_{\infty}^p). $$

 \end{enumerate}\qed
\end{theorem}
The following corollary is immediate from the conclusions of Theorems \ref{thm:maintheorem1} and \ref{thm:maintheorem2}.
\begin{corollary}\label{cor:maincor}
If \eqref{eq:lineq} satisfies Assumption \ref{as:assumption1} and $\underset{1\leq i\leq d}{\text{max}}\beta_i < 0$, then there exists $h^* > 0$ so that if $h$ is any sequence of step-sizes with $\|h\|_{\infty} \in (0, h^*)$, then $\underset{1\leq i\leq d}{\text{max}}\beta_i^A < 0$ and zero is a uniformly exponentially stable equilibrium of $x_{n+1} = \Phi^A(n;h)x_n$.  If \eqref{eq:lineq} sastisfies Assumption \ref{as:assumption2} and $\underset{1\leq i\leq d}{\text{max}}\mu_i < 0$, then there exists $h^* > 0$ so that if $h=\{h_n\}_{n=0}^{\infty}$ is any sequence of step-sizes with $\|h\|_{\infty} \in (0,h^*)$, then $\underset{1\leq i\leq d}{\text{max}}\mu_i^A < 0$ and zero is an exponentially stable equilibrium of $x_{n+1} =\Phi^A(n;h)x_n$.  \qed
\end{corollary}

\begin{example}\label{exmp:2dlin}
Consider the ODE \eqref{eq:2dlin}.  If we let $x = L(t)v$, then $\dot{v}= D(t)v \equiv [D_1+D_2(t)]v$ where the matix $D_1 = \left[\begin{array}{cc} \lambda_1 & \beta_0+a_1 \\ -a_1 & \lambda_2\end{array}\right]$ has distinct eigenvalues with real part $\lambda_1+\lambda_2 < 0$ and $D_2(t) = \left[\begin{array}{cc} 0 & \frac{\beta_0\cos(a_1 t)}{1+\beta_1 t^2} \\ 0 & 0\end{array}\right]$ with $K > 0$ so that $\|D_2(t)\|\leq K/(1+\beta_1 t^2)\rightarrow 0$.  Since $D$ has distinct eigenvalues there exists a time-independent change of variables $v = P z$ so that $P^{-1} D_1 P$ is diagonal.  Since $D_2(t)$ is integrable and $\dot{u} = D_1 u$ has an integral separation structure in a complex sense, it follows that $\dot{v} = (D_1+D_2(t))v$ has an integral separation structure in a complex sense.  Then the fact that $x = L(t)P z$ is a complex Lyapunov transformation implies that \eqref{eq:2dlin} has an integral separation structure.   

Once again consider the solution of \eqref{eq:2dlin} by the first order implicit Euler method with constant step-size $h_0 > 0$.  Since $A(t)$ is bounded and continuous and \eqref{eq:2dlin} and has an integral separation structure, Theorem \ref{thm:maintheorem2} implies that there exists $h^* > 0$ so that if $h_0 \in (0,h^*)$, then the endpoints of the Lyapunov spectrum of the discrete system \eqref{eq:impeuler} agree with those of the continuous system \eqref{eq:2dlin} to $\mathcal{O}(h_0)$ accuracy.  Corollary \ref{cor:maincor} implies that there exists $h^{**} \in (0,h^*)$ so that if $h_0 \in (0,h^{**})$, then the Lyapunov and Sacker-Sell spectrum of \eqref{eq:impeuler} are less than zero and the numerical solution is uniformly exponentially decaying for all sufficiently small $h_0 > 0$. \qed
\end{example}

We now discuss how to use the results of Theorem \ref{thm:maintheorem2} to characterize the approximate average exponential growth/decay rates of \eqref{eq:lineq} on the interval $[t,t+\Delta t]$. %attained by numerical solutions of \eqref{eq:lineq} on a time interval $[t,t+\Delta t]$. 
\begin{lemma}\label{lem:localgrowthcts}
Assume that the ODE \eqref{eq:lineq} satisfies Assumption \ref{as:assumption2}.  Let $X(t)$ be a fundamental matrix solution of \eqref{eq:lineq} and let $X(t) = Q(t)R(t)$ be a QR factorization where $Q(t) \in \mathbb{R}^{d \times d}$ is orthogonal and $R(t) \in \mathbb{R}^{d\times d}$ is upper triangular with positive diagonal entries.  The average exponential growth/decay rates of $X(t)$ on the interval $[t,t+\Delta t]$ where $t \geq t_0$ and $\Delta t > 0$ are given by the following Steklov averages:
\begin{equation}\label{eq:steklovavg}
s_i(t,\Delta t) = \dfrac{1}{\Delta t} \int_{t}^{t+\Delta t} B_{i,i}(\tau) d\tau, \quad i=1,\hdots,d.
\end{equation}
\end{lemma}
\begin{proof}
Let $t \geq t_0$ and $\Delta t > 0$.  Since $X(t)=Q(t)R(t)$ and $Q(t)$ is orthogonal the exponential growth/decay of $X(t)$ on $[t,t+\Delta t]$ is given by the exponential growth/decay of $R(t)$ on $[t,t+\Delta t]$.  We express $R(t+\Delta t) = R(t+\Delta t,t)R(t)$ where $R(\tau,t)$ is the unique and upper triangular solution of the matrix ODE IVP
$$\dot{\Phi} = B(\tau)\Phi, \quad \tau > t, \quad \Phi(t,t) = I_d$$
where $I_d$ is the $d \times d$ identity matrix.  We can express $R(t+\Delta t,t)$ as
$$R(t+\Delta t,t) = \text{diag}\left( e^{\Delta t s_1(t,\Delta t)},\hdots,e^{\Delta t s_d(t,\Delta t)} \right) + N(\tau,t).$$
Since \eqref{eq:lineq} satisfies Assumption \ref{as:assumption2}, Theorem 5.2 of \cite{DVV3} implies that
$$R(\tau,t) = \text{diag}\left( e^{\Delta t s_1(t,\Delta t)},\hdots,e^{\Delta t s_d(t,\Delta t)} \right) \left(I_d + \overline{N}(t+\Delta t,t)\right)$$
where $\|\overline{\mathcal{N}}(t+\Delta t,t)\| \leq K \|p(\Delta t)\|$ for some polynomial $p(z)$ with $p(0)= 0$.  Hence the average exponential growth/decay rates of $X(t)$ on $[t,t+\Delta t]$ are given by the quantities $s_i(t,\Delta t)$ for $i=1,\hdots,d$. \qed
\end{proof}
We can prove a result analagous to Lemma \ref{lem:localgrowthcts} for discrete systems \eqref{eq:dislineq} with a p-approximate integral separation structure where exponentials of Steklov averages are replaced by the products of diagonal entries of the upper triangular factor $R^A(n;h)$ in a discrete QR iteration.  Theorem \ref{thm:maintheorem2} then implies that for sufficiently small step-sizes we have $R^A_{i,i}(n;h) = e^{h_n s_i(t_n,h_n)} + \mathcal{O}(\|h\|^{p+1})$ and hence the average exponential growth/decay rate of fundamental matrix solutions of \eqref{eq:dislineq} are approximately (up to a term of the form $\mathcal{O}(\|h\|^{p+1}_{\infty})$) given by the Steklov averages \eqref{eq:steklovavg}.  It follows that for sufficiently small step-sizes the approximate average exponential growth/decay of a numerical solution of \eqref{eq:lineq} from time $t_n$ to $t_{n+k}$ for $k \geq 1$ is given by the average exponential growth/decay rate on the interval $[t_n,t_{n+k}]$ of the $d$ real-valued scalar test problems
$$\dot{y}_i = B_{i,i}(t)y_i, \quad i=1,\hdots,d.$$
This local-in-time stability argument is especially important for applying the linear stability theory to nonlinear ODE IVPs where $A(t) = Df(x(t;x_0,t_0),t)$ cannot be formed exactly. Regardless of the global error of $x_n$ from $x(t_n;x_0,t_0)$ we can still approximately quantify the average exponential growth/decay rates of the numerical solution on the next time interval $[t_n,t_{n+1}]$ assuming that $h_n$ is sufficiently small.

\subsection{Stability of the test problem}\label{sec:1dresults}

In this section we consider the numerical stability of a linear scalar test equation
\begin{equation}\label{eq:1dnalin}
\dot{z} = \lambda(t)z, \quad  t \geq t_0
\end{equation}
where $\lambda:(t_0,\infty)\rightarrow \mathbb{C}$ is $C^{p+1}$ and bounded with $\text{sup}_{t \geq t_0}|\lambda(t)| \leq M$ for some $M > 0$.  For full generality we consider the complex-valued case rather than the real-valued case justified in Section \ref{sec:linearresults}.  The numerical solution of \eqref{eq:1dnalin} by $\mathcal{M}$ using a sequence of step-sizes $h=\{h_n\}_{n=0}^{\infty}$ takes the form $z_{n+1} = \Phi^{\lambda}(n;h)z_n$ where $\Phi^{\lambda}(n;h) \in \mathbb{C}$.

\begin{theorem}\label{thm:1dcounterexample}
Suppose that $\mathcal{M}$ is a Runge-Kutta method with local truncation error of order $p \geq 1$ and $\lambda(t) \in \mathbb{R}$ for all $t \geq t_0$. Then, given any $h_0 > 0$ and any $-\alpha < 0$ we can find $\lambda(t)$ so that \eqref{eq:1dnalin} has Sacker-Sell spectrum with right endpoint given by $-\alpha$ and the numerical solution of $\dot{x}(t) = \lambda(t)x(t)$ using $\mathcal{M}$ with fixed step-size $h_0 > 0$ and initial condition $x(t_0) = x_0 \neq 0$ grows at an exponential rate.
\end{theorem}
\begin{proof}
Let $h_0 > 0$ and $-\alpha  < 0$ be given.  Let $\xi(\cdot)$ be the stability function of $\mathcal{M}$.  Since $\mathcal{M}$ has local truncation error of order $p \geq 1$ there exists $\delta > 0$ so that if $r\in (0,\delta)$, then $|\xi(1+r)| > 1$.  Let $D > \alpha$ be such that $D+\alpha \in (0,\delta)$, $\lambda(t) = D\cos(\frac{2\pi(t-t_0)}{h_0})- \alpha$, and note that the right endpoint of the Sacker-Sell spectrum of $\dot{x} = \lambda(t)x$ is $-\alpha$.  The numerical solution of \eqref{eq:1dnalin} with the method $\mathcal{M}$ using the fixed step-size $h_0$ is $x_{n+1} =\xi(h_0(D-\alpha))x_n$ and $\xi(h_0(D-\alpha))| > 1$.  It follows that $|x_n| \rightarrow \infty$ as $n \rightarrow \infty$ at a rate of $|\xi(h_0(D-\alpha))|^n$.\qed
\end{proof}

The geometric intuition for Theorem \ref{thm:1dcounterexample} is that time-dependent oscillations of $h\lambda(t)$ into and out of the linear stability domain of a method can trigger instabilities in the numerical solution.   The following proof of Theorems \ref{thm:maintheorem1} and \ref{thm:maintheorem2} for scalar ODEs of the form \eqref{eq:1dnalin} shows how we can control the accuracy of the Lyapunov and Sacker-Sell spectrum of the numerical solution using bounds on the local truncation error to guarantee exponential decay.
 
\begin{proof}[Proof of Theorems \ref{thm:maintheorem1} and \ref{thm:maintheorem2} in one dimension]\label{prf:1dprf}

Because the method $\mathcal{M}$ has local truncation error of order $p\geq 1$ and $\lambda(t)$ is bounded and $C^{p+1}$, there exists $h_1^* > 0$ so that if $h=\{h_n\}_{n=0}^{\infty}$ is any sequence of step-sizes with $\|h\|_{\infty} \in (0,h_1^*)$, then 
$$\Phi^{\lambda}(n;h) = \exp\left( \int_{t_n}^{t_{n+1}}\lambda(\tau)d\tau \right) + E(n;h)^{\lambda} \equiv I^{\lambda}(n;h) +E^{\lambda}(n;h) $$
where $E^{\lambda}(n;h) = K^{\lambda}(n;h) h_n^{p+1}$ and $\text{sup}_{n\geq 0}|K^{\lambda}(n;h)| \leq K^{\lambda}$ for some $K^{\lambda} > 0$.  If $n \geq m$ and $\|h\|_{\infty} \in (0,h_1^*)$, then
\begin{equation}\label{eq:1dthm1}
\prod_{k=n}^{m}\Phi^{\lambda}(k;h) =\left( \prod_{k=m}^{n}(1+E^{\lambda}(k;h)(I^{\lambda}(k;h))^{-1})\right)\exp\left(\int_{t_m}^{t_{n+1}}\lambda(\tau)d\tau\right) .
\end{equation}
%We remark that if $|a| < 1/2$, then 
%\begin{equation}\label{eq:1dthm2}
%0 < 1+a \leq e^{a} \text{ and } e^{-2|a|} \leq 1-|a|.
%\end{equation}
Let $h_2 \in (0,h_1^*]$ be so small that if $h=\{h_n\}_{n=0}^{\infty}$ is any sequence of step-sizes with $\|h\|_{\infty} \in (0, h_2^*)$, then $\text{sup}_{n \geq 0}\|E^{\lambda}(n;h)(I^{\lambda}(n;h))^{-1}\| \leq K^{\lambda} \|h\|_{\infty}^{p+1}e^{M\|h\|_{\infty}} < 1/2$.  So, if $\|h\|_{\infty} \in (0, h_2^*)$, and $n \geq m \geq 0$, then \eqref{eq:1dthm1} implies that %and \eqref{eq:1dthm2} imply that  
\begin{multline}\label{eq:1dthm3}
\small
 \exp\left(\int_{t_m}^{t_{n}}\lambda(\tau)d\tau - 2\sum_{k=m}^{n}|E^{\lambda}(k;h)(I^{\lambda}(k;h))^{-1}|\right) \leq \prod_{k=n}^{m}\|\Phi^{\lambda}(k;h)\|   \leq \\
\exp\left(\int_{t_m}^{t_{n+1}}\lambda(\tau)d\tau + \sum_{k=m}^{n}E^{\lambda}(k;h)(I^{\lambda}(k;h))^{-1}\right).
\end{multline}
For any sequence of step-sizes $h$ with $\|h\|_{\infty} \in (0, h_2^*)$ we have
\begin{equation}\label{eq:1dthm4}|\sum_{k=m}^{n}E^{\lambda}(k;h)(I^{\lambda}(k;h))^{-1}| \leq\sum_{k=m}^{n} K^{\lambda} h_k^{p+1}e^{M\|h\|_{\infty}^p} \leq K^{\lambda}e^{M\|h\|^p}\|h\|_{\infty}^p(t_{n+1}-t_m).
\end{equation}
If $\|h\|_{\infty} \in (0,h_2^*)$, then the conclusions of Theorem \ref{thm:maintheorem2} follow from inequalities \eqref{eq:1dthm3} and \eqref{eq:1dthm4}.  The conclusion of Theorem \ref{thm:maintheorem1} follows by letting $\varepsilon > 0$ be given and then setting $h^*$ to be so small that if $\|h\|_{\infty} \in (0, h^*)$, then $K^{\lambda}e^{M\|h\|_{\infty}^p}\|h\|_{\infty}^p < \varepsilon/2$.\qed
\end{proof}

From experience it is clear that certain subsets of A-stable Runge-Kutta methods, such as AN- and L-stable methods, have superior stability properties compared to other classes of implicit and explicit methods.  For an AN-stable Runge-Kutta method $\mathcal{M}$, if $\text{Re}(\lambda(t)) \leq 0$ on some interval $[\tau_1,\tau_2]$, then $|\Phi^{\lambda}(n;h)| \leq 1$ whenever $t_n,t_{n+1} \in [\tau_1,\tau_2]$.  We extend this type of analysis to methods that are not AN-stable.  Fix a step-size sequence $h = \{h\}_{n=0}^{\infty}$ and for each $n \geq 0$ consider the following associated mean autonomous ODE:
\begin{equation}\label{eq:1dnalinmean}
\dot{w} = \xi_n w, \quad \xi_n :=\xi(\lambda; h_n) = \frac{1}{h_n}\int_{t_n}^{t_n+h_n}\lambda(\tau)d\tau.
\end{equation}
Suppose that the approximate solution of \eqref{eq:1dnalin} at time $t_n$ is given by $z_n$.  Then the exact solutions of \eqref{eq:1dnalin} and \eqref{eq:1dnalinmean} with the initial condition $z(t_n) = z_n$ are the same:
$$w(t_n+h_n) = z(t_n+h_n) = \exp\left(\int_{t_n}^{t_n+h_n}\lambda(\tau)d\tau\right)z_n.$$
The solution of \eqref{eq:1dnalin} and \eqref{eq:1dnalinmean} by $\mathcal{M}$ using the step-size $h_n$ are then given by
$$w(t_n+h_n) \approx w_{n+1} = \Psi(h_n\xi_n)w_n, \quad z(t_n+h_n) \approx z_{n+1} = \Phi^{\lambda}(n;h)z_n.$$
Since the exact solutions are the same, there exists $h^* > 0$ so that if $\|h\| \in (0, h^*)$, then 
\begin{equation}\label{eq:stabfcn}
\Phi^{\lambda}(n;h) = \Psi(h_n\xi_n) + \mathcal{O}(h_n^{p+1}), \quad n \geq 0.
\end{equation}

Equation \eqref{eq:stabfcn} implies the following theorem.

\begin{theorem}\label{thm:stabfcntheorem}
Let $S$ be the linear stability region of a Runge-Kutta method $\mathcal{M}$ and let $h_n > 0$.  For each $\varepsilon \in (0,1)$ define $S'(\varepsilon) =\{z\in S: |\Psi(z)| \leq 1-\varepsilon\}$.  If there exists $\varepsilon \in (0,1)$ so that $\int_{t_n}^{t_n+h_n}\lambda(\tau)d\tau \in S'(\varepsilon)$ and $|\Phi^{\lambda}(n;h) - \Psi(h_n\xi_n)| \leq \varepsilon$, then $|\Phi^{\lambda}(n;h)| \leq 1$.  If there exists $\overline{h} > 0$ and $\varepsilon \in (0,1)$ so that $h_n \lambda(t_n+h_n) \in S'(\varepsilon)$ for all $h_n \in (0,\overline{h})$, then there exists $h^* \in (0,\overline{h}]$ so that if $h_n \in (0,h^*)$, then $|\Phi^{\lambda}(n;h)| \leq 1$.
\qed
\end{theorem}

We close this section by remarking that we cannot extend equation \eqref{eq:stabfcn} in a straightforward way to higher-dimensional problems since for $d \geq 2$ the matrix exponential function $e^{\int_{t_n}^{t}A(\tau)d\tau}$ is not in general a solution of \eqref{eq:lineq} for $t \geq t_n \geq t_0$.  It is necessary to employ a time-dependent change of variables to reduce the analysis of \eqref{eq:lineq} to a scalar test problem of the form of \eqref{eq:1dnalin}.

\subsection{Proof of the main results for linear ODEs}\label{sec:linearresultsproofs}

Let $X(t)$ be any fundamental matrix solution of \eqref{eq:lineq} and let $X(t) = Q(t)R(t)$ be a QR factorization where $Q(t)$ is orthogonal and $R(t)$ is upper triangular with positive diagonal entries.  Without loss of generality we can assume that $Q(t)$ is the orthogonal matrix of the corresponding upper triangular ODE \eqref{eq:lineqB}.  For each $n \geq 0$ let the transition matrix $X(t,t_n)$ be the unique $d \times d$ matrix solution of the following matrix ODE IVP:
\begin{equation}\label{eq:Atrans}
\left\{
\begin{array}{lcr}
 \dot{\Psi}(t) = A(t) \Psi(t),  \quad t > t_{n}\\
 \Psi(t_{n}) = I_{d}
 \end{array}
 \right.
\end{equation}
where $\Psi(t) \in \mathbb{R}^{d\times d}$ and $I_d$ is the $d \times d$ identity matrix.  We can factor $X(t_n)$ as
 $$X(t_n) = X(t_n,t_{n-1})\cdot \hdots \cdot X(t_1,t_0)X(t_0).$$
Similarly for each $n \geq 0$ we let $R(t,t_n)$ be the unique $d\times d$ matrix solution of the following matrix ODE IVP:
\begin{equation}\label{Btrans}
\left\{
\begin{array}{lcr}
 \dot{\Phi}(t) = B(t) \Phi(t), \quad t > t_{n}\\
 \Phi(t_{n}) = I_d
 \end{array}
 \right.
\end{equation}
where $\Phi(t) \in \mathbb{R}^{d \times d}$.  We can then factor $R(t_n)$ as 
$$R(t_n) = R(t_n,t_{n-1}) \cdot \hdots \cdot R(t_1,t_0)R(t_0).$$
Notice that we have $X(t,t_{n}) =Q(t)R(t,t_n)Q(t_n)^T$ for $n \geq 0$. The local error equations
\begin{equation}\label{eq:truncation_error}
\Phi^A(n;h)= X(t_{n+1},t_n) + E^A(n;h), \quad \Phi^B(n;h) = R(t_{n+1},t_n)+E^B(n;h)
\end{equation}
and the definition $\overline{F}(h;n) := -E^B(n;h)+Q(t_{n+1})^T E^A(n;h) Q(t_n)$ imply that
\begin{equation}\label{eq:phiABrelation}
\Phi^A(n;h) = Q(t_{n+1})(\Phi^B(n;h)+\overline{F}(n;h))Q(t_n)^T, \quad 
\end{equation}
where $\|\overline{F}(n;h)\| \leq L(\|E^A(n;h)\| + \|E^B(n;h)\|)$ for some $L > 0$ since $Q_n$ and $Q(t_n)$ are orthogonal.  By the assumption made in Section \ref{sec:mainresults}, $\|h\|_{\infty}$ is always such that $\Phi^A(n;h)$ is invertible for all $n \geq 0$ so we can let $Q_0 :=Q(t_0)$ and inductively form QR factorizations 
\begin{equation}\label{eq:phiAdisQR}
\Phi^A(n;h) Q_n = Q_{n+1}R^A(n;h), \quad n \geq 0
\end{equation}
where $Q_n$ is orthgonal and $R^A(n;h)$ is upper triangular with positive diagonal entries for all $n \geq 0$.  Combining \eqref{eq:phiABrelation} and \eqref{eq:phiAdisQR} results in the equation
\begin{equation}\label{eq:RAphiB}
R^A(n;h) = Q_{n+1}^T Q(t_{n+1})(\Phi^B(n;h)+\overline{F}(n;h))Q(t_n)^T Q_n.
\end{equation}

The Lyapunov and Sacker-Sell spectra of $v_{n+1} = (\Phi^B(n;h)+\overline{F}(n;h))v_n$ and $x_{n+1} = \Phi^A(n;h)x_n$ coincide since $x_n = Q(t_n) y_n$ is a discrete Lyapunov transformation.

\begin{proof}[Proof of Theorem \ref{thm:maintheorem1}] 
By the estimates \eqref{eq:1dthm3} and \eqref{eq:1dthm4} in the proof of Theorems \ref{thm:maintheorem1} and \ref{thm:maintheorem2} in one dimension, there exists $h_1^* > 0$ so small that if $h$ is any sequence of step-sizes with $\|h\|_{\infty} \in (0, h_1^*)$, then 
$$\beta_i^B = \beta_i + \mathcal{O}(\|h\|_{\infty}^p) \text{ and } \alpha_i^B = \alpha_i + \mathcal{O}(h_{\text{max}}^{p}).$$

Let $\varepsilon > 0$ be given.  By continuity of the Sacker-Sell spectrum there exists $\delta > 0$ so that if $\|\overline{F}(n;h) \| < \delta$, then the endpoints of the Sacker-Sell spectrum of $v_{n+1} = (\Phi^B(n;h)+\overline{F}(n;h))v_n$ (and hence of $x_{n+1} = \Phi^A(n;h)x_n$) satisfy
$$|\alpha_i^A-\alpha_i| < \varepsilon \text{ and } |\beta_i^A-\beta_i| < \varepsilon, \quad i=1,\hdots,d.$$
We can always bound $\|\overline{F}(n;h)\| < \delta$ as follows.  Since $\mathcal{M}$ has local truncation error of order $p \geq 1$, we can choose $h_2^* \in (0,h_1^*]$ be so small that if $\|h\|_{\infty} < h_2^*$, then $\|\overline{F}(n;h)\| = \mathcal{O}(h_{n}^{p+1}) = \mathcal{O}(\|h\|_{\infty}^{p+1})$.  Then we can choose $h^* \in (0,h_2^*]$ so small that $\|\overline{F}(n;h)\| < \delta$.\qed
\end{proof}

We assume for the remainder of this section that \eqref{eq:lineq} satisfies Assumption \ref{as:assumption2}.  The proof of Theorem \ref{thm:maintheorem2} is accomplished using several technical lemmas.
 \begin{lemma}\label{DISpreserve}
There exists $h^* > 0$ so that if $h=\{h_n\}_{n=0}^{\infty}$ is any sequence of step-sizes with $\|h\|_{\infty} \in (0, h^*)$, then the system $y_{n+1} = \Phi^B(n;h)y_n$ has a p-approximate discrete integral separation structure.
 \end{lemma}
 \begin{proof}
Given any sequence of step-sizes $h=\{h_n\}_{n=0}^{\infty}$, the diagonal entries $\Phi^B_{i,i}(n;h)$ are such that $y_{n+1}^i = \Phi^B_{i,i}(n;h)y_n^i$ are approximations to the scalar ODE $\dot{y}_i = B_{i,i}(t)y_i$ with $y_i(t_0) = y_0^i$ using the method $\mathcal{M}$.  Because $\mathcal{M}$ has local truncation of order $p$ and $B(t)$ is bounded and $C^{p+1}$, there exists $h_1^* > 0$ so that if $h$ is such that $\|h\|_{\infty} < h_1^*$, then
$$\Phi^B_{l,l}(n;h) = \exp\left(\int_{t_n}^{t_{n+1}}B_{l,l}(\tau)d\tau\right) + E_{l}^B(n;h) \equiv I_l^B(n;h)+E_{l}^B(n;h), \quad n \geq 0$$
where $E_{l,l}^B(n;h) = K_l^B(n;h)h_n^{p+1}$ and $\text{sup}_{n \geq 0} K_l^B(n;h) \leq K_l^B < \infty$ for $l=1,\hdots,d$.  There exists $h_2^* > 0$ with $h_2^* \in (0,h_1^*]$ so that if $h$ is such that $\|h\|_{\infty} \in (0,h_2^*)$, then we have $\text{inf}_{n \geq 0}\Phi^B_{l,l}(n;h) > 0$ for $l=1,\hdots,d$ and therefore if $n\geq m$ and $i > j$, then
\begin{equation}\label{eq:DISpreserve1}\prod_{k=n}^{m}\frac{\Phi^B_{i,i}(k;h)}{\Phi_{j,j}^B(k;h)} = e^{\int_{t_m}^{t_{n+1}}B_{i,i}(\tau)-B_{j,j}(\tau)d\tau}\left[\prod_{j=n}^{m} \frac{1+E_i^B(k;h)(I_i^B(k;h))^{-1}}{1+E_j^B(k;h)(I_j^B(k;h))^{-1}}\right].
\end{equation}
%Note that for $q\in (-1/2,1/2)$ and we have the following inequalities
%\begin{equation}\label{eq:DISpreserve2}
%e^{-2|q|} \leq 1-|q|, \quad e^{-q} \leq \dfrac{1}{1+|q|}, \quad \dfrac{1}{1-|q|} \leq e^{2|q|}, 1+|q| %\leq e^{|q|}
%\end{equation}
Note that since $\text{inf}_{n \geq 0}\Phi^B_{l,l}(n;h) > 0$ for $l=1,\hdots,d$ and $\|h\| \in (0,h_2^*)$ it follows that $\{\Phi^B_{l,l}(n;h)^{-1}\}_{n=0}^{\infty}$ is uniformly bounded for $l=1,\hdots,d$. Since $B(t)$ is bounded, for $i=1,\hdots,d$ there exists $M_i^B > 0$ so that $\text{sup}_{t \geq t_0}|B_{i,i}(t)| \leq M_i^B$.  Therefore, there exists $h_3^* \in (0,h_2^*]$ so that if $\|h\|_{\infty} \in (0, h_3^*)$, then  
\begin{equation}\label{eq:DISpreserve2.2}
\text{sup}_{n \geq 0}|E_l^B(n;h)I_l^B(n;h)^{-1}| \leq K_l^B\|h\|^{p+1}_{\infty} e^{\|h\|_{\infty} M_l^B} < 1/2, \quad l=1,\hdots,d.
\end{equation}

 Assumption \ref{as:assumption2} implies that if $i > j$, then $B_{i,i}$ and $B_{j,j}$ satisfy \eqref{eq:IS} or \eqref{eq:nonIS}.  Let $IS$ be the set of all pairs of integers $(i,j)$ with $1\leq i,j \leq d$ and $i > j$ so that $B_{i,i}$ and $B_{j,j}$ satisfy \eqref{eq:IS}.  If $(i,j) \in IS$, then \eqref{eq:IS}, \eqref{eq:DISpreserve1}, \eqref{eq:DISpreserve2.2}, and $\|h\|_{\infty} \in (0, h_3^*)$ imply that if $n \geq m$, then
\begin{multline*}
\prod_{k=n}^{m}\Phi^B_{i,i}(k;h)(\Phi_{j,j}^B(k;h))^{-1} \geq e^{a_{i,j}(t_{n+1}-t_m)+b_{i,j}}e^{-\sum_{k=m}^{n}(2K_i^B e^{\|h\|_{\infty}M_i^B}+K_j^B e^{\|h\|_{\infty}M_j^B})h_k^{p+1}} \geq\\ \exp\left((a_{i,j}-(2K_i^B e^{\|h\|_{\infty}M_i^B}+K_j^B e^{\|h\|_{\infty}M_j^B})\|h\|_{\infty}^p )(t_{n+1}-t_m)+b_{i,j}\right).
\end{multline*}
Let $h^* \in (0,h_3^*]$ be such that if $\|h\|_{\infty} \in (0, h^*)$, then
\begin{equation}\label{eq:DISpreserve3}
a_{i,j}-(2K_i^B e^{\|h\|_{\infty}M_i^B}+K_j^B e^{\|h\|_{\infty}M_j^B})\|h\|_{\infty}^p  > 0
\end{equation}
for all $(i,j) \in IS$.  It then follows that if $(i,j) \in IS$ and $\|h\|_{\infty} \in (0, h^*)$, then $\Phi^B_{i,i}(n;h)$ and $\Phi^B_{j,j}(n;h)$ satisfy satisfy an inequality of the form \eqref{eq:adiss1}.

If $(i,j) \notin IS$ so that $B_{i,i}$ and $B_{j,j}$ satisfy \eqref{eq:nonIS}, then \eqref{eq:DISpreserve1}, \eqref{eq:DISpreserve2.2}, and $\|h\|_{\infty} < h_3^*$ imply that given $\varepsilon > 0$, there exists $M_{i,j}(\varepsilon)$ so that if $n \geq m$, then 
\begin{multline*}
\prod_{k=n}^{m}\Phi^B_{i,i}(k;h)(\Phi_{j,j}^B(k;h))^{-1} \leq \exp\left(M_{i,j}+\frac{\varepsilon(t_{n+1}-t_m)}{2}\right)\exp\left(\sum_{k=m}^{n}(K_i^B+K_j^B)h_k^{p+1}\right)   \\
\leq \exp\left(M_{i,j}+(\varepsilon+(2K_i^B+K_j^B)\|h\|_{\infty}^p)(t_{n+1}-t_m)\right).
\end{multline*}
Similarly, if $\|h\|_{\infty} \in (0, h^*)$, then 
$$\prod_{k=n}^{m}\Phi^B_{i,i}(k;h)(\Phi_{j,j}^B(k;h))^{-1}  \geq \exp\left(-M_{i,j}-(\varepsilon+(2K_i^B+K_j^B)\|h\|_{\infty}^p)(t_{n+1}-t_m)\right)$$
 and it then follows that \eqref{eq:adiss2} is satisfied whenever $\|h\|_{\infty} \in (0, h^*)$.  Therefore, if $\|h\|_{\infty} < h^*$, then $\text{inf}_{n\geq 0}\Phi^B_{i,i}(n;h) > 0$ for $n \geq 0$ and $i=1,\hdots,d$ and conditions \eqref{eq:adiss1} and \eqref{eq:adiss2} are satisfied.  It follows that $y_{n+1} = \Phi^B(n;h)y_n$ has a p-approximate discrete integral separation structure.\qed
\end{proof}

The size that $h^* > 0$ must be taken in Lemma \ref{DISpreserve} depends on the integral separation through the inequality \eqref{eq:DISpreserve3} and on the strength of growth/decay by enforcing that $\text{inf}_{n \geq 0}\Phi^B_{i,i}(n;h) > 0$ for $i=1,\hdots,d$.  Stronger integral separation between diagonal elements of $B(t)$  (i.e. larger values of $a_{i,j}$) and larger values of $|\int_{t_n}^{t_{n+1}}B(\tau)d\tau)|$ require the smaller step-sizes to ensure the discrete system inherits these properties.

  \begin{lemma}\label{diagest}
There exists $h^* > 0$ so that if $\|h\|_{\infty} \in (0, h^*)$, then $J(n;h):=\|R^A(n;h)-\Phi^B(n;h)\| = \mathcal{O}(\|h\|_{\infty}^{p+1})$.
\end{lemma}
\begin{proof} 
Using Lemma \ref{DISpreserve} we can find $h_1^* > 0$ such that if $\|h\|_{\infty} < h_1^*$, then $\Phi^B(n;h)$ has a p-approximate discrete integral separation structure and so that $F(n;h) := -Q(t_{n+1})\overline{F}(n;h) Q(t_n)$ with $F(n;h) = \mathcal{O}(h_n^{p+1})$. Theorem \ref{thm:disQR} implies that there exists an $h^* \in (0,h_1^*]$, $K > 0$, and a sequence $\{\tilde{Q}_n\}_{n=0}^{\infty}$ with each $\tilde{Q}_n \in \mathbb{R}^{d\times d}$ orthogonal so that if $\|h\|_{\infty} < h^*$, then 
$$ \tilde{Q}_{n+1} R^A(n;h) = (\Phi^B(n;h) + E(n;h))\tilde{Q}_n, \quad \|\tilde{Q}_n - I\| \leq K G$$
where $E(n;h) = Q_{n+1}^T F(n;h) Q_n$ and $G = \text{sup}_{n\geq 0 } \{\|F(n;h)\|,\|E(n;h)\| \} = \mathcal{O}(h_n^{p+1})$.  It then follows that if $\|h\|_{\infty} < h^*$, then $J(n;h) = R^A(n;h)-\Phi^B(n;h) = \mathcal{O}(\|h\|_{\infty}^{p+1})$.\qed

\end{proof}

\begin{lemma}\label{DISpreserve2}
Suppose that the ODE \eqref{eq:lineq} satisfies Assumption \ref{as:assumption2}. Then, there exists $h^* > 0$ so that if $\|h\|_{\infty} \in (0,h^*)$, then $u_{n+1} = R^A(n;h)u_n$ has a p-approximate integral separation structure.
 \end{lemma}
 \begin{proof}
Combine Lemma \ref{diagest} and the method used to prove Lemma \ref{DISpreserve}.\qed
 \end{proof}
We now complete the proof of Theorem \ref{thm:maintheorem2}.  Let $h^* > 0$ be so small that if $\|h\|_{\infty} \in (0, h^*)$, then the conclusions of Lemmas \ref{DISpreserve}, \ref{diagest}, and \ref{DISpreserve2} and Theorem \ref{thm:lediscrete} hold.  The conclusions of Theorem \ref{thm:maintheorem2} are proved by combining the conclusions of Lemmas \ref{DISpreserve}, \ref{diagest}, \ref{DISpreserve2} and the conclusions of Theorems \ref{thm:ssdiss} and \ref{thm:lediscrete}. \qed

%=======================================================================================================================================
%=======================================================================================================================================
%=======================================================================================================================================
%=======================================================================================================================================
%=======================================================================================================================================
%=======================================================================================================================================

\section{Nonlinear stability}\label{sec:nonlin}

Let $\mathcal{M}$ be a one-step method with local truncation error of order $p \geq 1$.  Consider the ODE \eqref{eq:ode} and assume that $f\in C^k(\mathbb{R}^d \times (s,\infty),\mathbb{R}^d)$ for some integer $k \geq \text{max}\{3,p+1\}$.  For the remainder of this section we assume that $x(t;x_0,t_0)$ is a bounded solution of \eqref{eq:ode} with initial condition $x(t_0) = x_0$ and also that the right end-point of the Sacker-Sell spectrum of $\dot{u} = Df(x(t;x_0,t_0),t)u \equiv A(t)u$ of $x(t;x_0,t_0)$ is $-\alpha < 0$ so that $x(t;x_0,t_0)$ is uniformly exponentially stable.  We also assume that there exists $h^* > 0$ so that if $\|h\|_{\infty} \in (0,h^*)$, then the local truncation error of $\mathcal{M}$ applied to solve any IVP of \eqref{eq:ode} takes the form $C x^{(p+1)}(\xi_n;x_n,t_n) h_n^{p+1}$ where $x^{(p+1)}$ is the $(p+1)^{st}$ time-derivative, $\xi_n \in (t_n,t_{n+1})$, and $C \in \mathbb{R}$.

Fix $u_0 \in \mathbb{R}^d$ and $t \geq s \geq t_0$.  Let $u(t):=x(t;u_0,s)$ and $x(t):=x(t;x_0,t_0)$.  By Taylor expanding $f(u(t),t)$ about $x(t)$ we can express $u(t)$ as
\begin{equation}\label{eq:nonlin.1}
u(t) = X(t)X^{-1}(s)u_0 + \int_{s}^t X(t)X(\tau)^{-1}(b(\tau) + R(u_0,s,\tau))d\tau,
\end{equation}
\begin{equation*}
b(t) := f(x(t),t)-Df(x(t),t)x(t) \equiv f(x(t),t)-A(t)x(t)
\end{equation*}
\begin{equation*}
R(u_0,s,t) :=
 (u(t)-x(t))^T Hf(a(u_0,t_0,t,s),t)(u(t)-x(t),t_0))
\end{equation*} 
and where $Hf$ is the Hessian of $f(x,t)$ with respect to $x$ and $a(u_0,t_0,t,s)$ is some point on the line segment between $u(t)$ and $x(t)$.  Boundedness, uniform exponential stability of $x(t;x_0,t_0)$, and the fact that $f \in C^k(\mathbb{R}^d \times (t_0,\infty),\mathbb{R}^{d})$ for $k \geq 3$ imply that there exists $\delta_M,M > 0$ so that if $\|u_0-x_0\| < \delta_M$ and $t \geq s \geq t_0$, then 
\begin{multline}\label{eq:derivativebounds}
\text{max}\{\|x^{(p+1)}(t;u_0,s)\|,\|\frac{\partial }{\partial u}x(t;u,s)|_{u=u_0}\|, \\
\|Hf(a(u_0,s,t),t)\|,\|\frac{\partial }{\partial u}R(u,s,t)\|,\|R(u,s,t)\|\|\}\leq M.
\end{multline}

\begin{theorem}\label{thm:nonlin1} Consider the numerical solution $u_n = u(n;u_0,t_0,h)$ generated by approximating $x(t;u_0,t_0)$ with the method $\mathcal{M}$ using the initial condition $u_0$ at initial time $t_0$ and let $x_n :=x(t_n;u_0,t_0)$ for $n \geq 0$. Given $D > 0$, there exists $h^* > 0$ so that for each sequence of step-sizes $h$ with $\|h\|_{\infty} \in (0,h^*)$ there exists $\delta > 0$ so that if $\|u_0-x_0\| < \delta$, then $\text{sup}_{n \geq 0}\|u_n-x_n\| \leq D \|h\|_{\infty}^{p}$. 
\end{theorem}
\begin{proof}
By \eqref{eq:nonlin.1} and the assumption on the form of the local truncation error, there exists $h_1^* > 0$ so that if $\|h\|_{\infty} \in (0, h_1^*)$, then we have for some $\xi(h_n) \in (t_n,t_{n+1})$:
\begin{multline}\label{eq:nonlin1.1}
u_{n+1} -x_{n+1} = X(t_{n+1})X(t_n)^{-1}(u_n-x_n)\\
 + \int_{t_n}^{t_{n+1}}X(t_{n+1})X(\tau)^{-1}R(u_n,t_n,\tau)d\tau + C x^{(p+1)}(\xi(h_n);u_n,t_n)h_n^{p+1}
\end{multline}

Theorem \ref{thm:maintheorem1} implies that there exists $h_2^*, K, E^A > 0$ and $\gamma \in (0,\alpha)$ so that if $\|h\|_{\infty} < h_2^*$, then $X(t_{n+1})X(t_n)^{-1} = \Phi^A(n;h) + E_n^A h_n^{p+1}$ where $\text{sup}_{n \geq 0}\|E_n^A\| \leq E^A$ and so that if $n \geq m$, then $\|\prod_{k=n}^{m}\Phi^A(n;h)\| \leq Ke^{-\gamma(t_n-t_m)}$.  

If $y_n : = u_n - x_n$, then $\{y_n\}_{n=0}^{n}$ satisfies a difference equation of the form $y_{n+1} = a_n y_n + b_n$ where $a_n = \Phi^A(n;h)$ and $b_n$ is defined as the remainder of the right-hand side of \eqref{eq:nonlin1.1}.  The discrete variation of parameters formula implies that if $n > 0$, then
\begin{equation}\label{eq:nonlin1.2}
y_n = \left[\prod_{k=n-1}^{0}a_k\right]y_0 + \sum_{i=0}^{n-1}\left[\prod_{k=n-1}^{i+1}a_k\right] b_i.
\end{equation}
Let $K_M > 0$ be such that $\|X(t_{n+1})X(\tau)^{-1}\| \leq K_M$ for all $\tau \in [t_n,t_{n+1}]$ and $n \geq 0$.  Assume $D > 0$ be given and let $h^* \in (0,h_2^*]$ be such that $D(h^*)^{p} < \text{min}\{\delta_M,1\}$ and 
\begin{equation}\label{eq:nonlin1.3}
\left(E^A D (h^*)^{p} + K_M M  + |C|M (h^*)^p \right)h^* < \frac{ D(1-e^{-\gamma h_{\text{min}}})}{2(K+1)e^{\gamma h_{\text{min}}}}.
\end{equation}

For each sequence of step-sizes with $\|h\|_{\infty} < h^*$ we let $\delta := \text{min}\{\delta_M, \frac{D\|h\|_{\infty}^p}{2(K+1)}\}$.  We now complete the proof with an induction on $n$.  The base case $n=0$ is satisfied since $\|y_0\| < \delta \leq D\|h\|_{\infty}^p$.  Now assume that $\|y_i\| \leq D \|h\|_{\infty}^p$ for $i=0,\hdots,n-1$.  Then $D\|h\|_{\infty}^{p} < \text{min}\{\delta_M,1\}$ implies that $\|y_i\|^2 \leq\|y_i\| < \delta_M$ for $i=0,\hdots,n-1$.  Using \eqref{eq:nonlin1.3} we can bound $b_i$ as
$$\|b_i\| \leq \left(E^A  \|h\|_{\infty}^{p}\|y_i\|+ K_M M +|C|M\|h\|_{\infty}^{p}\right)\|h\|_{\infty}, \quad i=0,\hdots,n-1$$
and therefore by \eqref{eq:nonlin1.3} and the induction hypothesis, for $i=0,\hdots,n-1$ we have
$$\|b_i\| \leq \left(E^A D\|h\|_{\infty}^{2p} + K_M M  + |C|M\|h\|_{\infty}^p  \right)\|h\|_{\infty} < \frac{ D\|h\|_{\infty}^p(1-e^{-\gamma h_{\text{min}}})}{2(K+1)e^{\gamma h_{\text{min}}}}$$
Then $\|y_0\| < \delta \leq D\|h\|_{\infty}^p/(2(K+1))$, inequality \eqref{eq:nonlin1.2}, and $\text{inf}_{n \geq 0}h_n > 0$ imply that  
$$\|y_n\| \leq K\|y_0\| + D\|h\|_{\infty}^p/2 \leq D\|h\|_{\infty}^p.$$\qed

\end{proof}

\begin{theorem}\label{thm:nonlin2}  Assume that $f \in C^{k}(\mathbb{R}^d \times (t_0,\infty),\mathbb{R}^{d})$. Let $u_n :=u(n;u_0,t_0,h)$ and $w_n :=u(n;x_0,t_0,h)$ denote the numerical approximations of $x(t;u_0,t_0)$ and $x(t;x_0,t_0)$ generated by the method $\mathcal{M}$ with the sequence of step-sizes $h=\{h_n\}_{n=0}^{\infty}$.  Given any $\gamma \in (0,\alpha)$, there exists $h^* > 0$, $K > 0$, and $\delta > 0$ so that if $h$ is a sequence of step-sizes with $\|h\|_{\infty} \in (0, h^*)$ and $n \geq m \geq 0$, then $\|u_n-w_n\| \leq K e^{-\gamma(t_n-t_m)}\|u_m-w_m\|$.
\end{theorem}
\begin{proof}
As in the derivation of Equation \eqref{eq:nonlin1.1}, there exists $h_1^*  > 0$ so that if $\|h\|_{\infty} \in (0 ,h^*)$, then
\begin{multline}\label{eq:nonlin2.1}
u_{n+1} -w_{n+1} = X(t_{n+1})X(t_n)^{-1}(u_n-w_n) \\
+ \int_{t_n}^{t_{n+1}}X(t_{n+1})X(\tau)^{-1}(R(u_n,t_0,\tau)-R(w_n,t_0,\tau))d\tau +\\
 C( (x^{(p+1)}(\xi(u_n,t_n,h_n);u_n,t_n)-x^{(p+1)}(\xi(w_n,t_n,h_h);u_n,t_n))h_n^{p+1}.
\end{multline}
The result then follows by combining Theorem \ref{thm:nonlin2} with the bounds \eqref{eq:derivativebounds} and the discrete Gronwall inequality.\qed

\end{proof}

%=======================================================================================================================================
%=======================================================================================================================================
%=======================================================================================================================================
%=======================================================================================================================================
%=======================================================================================================================================
%=======================================================================================================================================

\section{Applications}\label{sec:numexp}

In this section we apply the theoretical results from Sections \ref{sec:mainresults} and \ref{sec:nonlin} to develop a time-dependent stiffness indicator and a one-step method that switches between implicit and explicit Runge-Kutta methods.  We denote Runge-Kutta methods by RK($\nu$-$p$-$\hat{p}$) where RK is an identifying string, $\nu$ is the number of stages, $p$ is the order of the method, and $\hat{p}$ is the order of the embedded method.  We give Butcher tableaux and discuss the accuracy and stability properties of the methods we use in the Appendix.

All the experiments in this section were conducted using a solver $odeqr$ implemented in MATLAB. This solver forms an approximate solution using a Runge-Kutta method with the capability of switching between different methods at each step.  In $odeqr$ the step-size is either constant or adaptive where an initial step-size guess is reduced by increments of $25\%$ until a tolerance is satisfied.  For an implicit method $odeqr$ solves the nonlinear stage value equations using Newton's method with an option for using exact and inexact Jacobians using the previous solution step as initial guess and an error tolerance of $10^{-12}$.

\subsection{Test ODEs}\label{sec:testode}

In this section we discuss the four ODEs used in our experiments in Sections \ref{sec:stiffdetect} and \ref{sec:qrimex}.  The first ODE we consider is Equation \eqref{eq:2dlin} with $\lambda_1 = 0.1$, $\lambda_2=-0.2$, $\beta_0 = 1000.0$, $\beta_1 = 0.001$, $a_1=a_2 = 2\pi$, and initial condition $x(0) = (1,-1)^T$.

 The second ODE we consider is the so-called "compost bomb instability".  We use the form of the equations appearing on page 1245 of \cite{WALC2011}:
\begin{equation}\label{eq:compbomb}
\left\{ \begin{array}{lcr}
\dot{T} = Cre^{\alpha T} - \frac{\lambda}{A}(T-T_a) \\
\dot{C} = \Pi - Cr \exp(\alpha T)\\
\dot{T}_a = \nu
\end{array}
\right.
\end{equation}
The parameters represent various physically and biologically relevant quantities with $r =0.01$, $\alpha =\ln(2.5)/10$, $\lambda = 5.049 \times 10^6$, $A = 3.9\times 10^7$, $\Pi = 1.055$, and $\varepsilon = 0.064$.  As in \cite{WALC2011} we use a sixth order approximation to  $e^{\alpha T}$.  We use the initial condition $T(0)=8.15$, $C(0) = 50.0$, $T_a(0) = 0.0$ and parameter values $\nu = 0.09$ and $\nu = 0.3$.  For these parameters and this initial condition the solution of \eqref{eq:compbomb} exhibits single- and double-spike excitable responses which generate time-intervals over which the solution is very stiff \cite{WALC2011}.  

The third equation we consider is the forced Van der Pol equation \cite{VDP} expressed as a first order ODE in two dimensional phase space:

\begin{equation}\label{eq:navdp2d}
\left\{ \begin{array}{lcr}
\dot{x_1}=  \mu(1-x_1^2)x_2 + x_1 - A \sin(\omega t)\\
\dot{x_2} = x_1
\end{array}
\right.
\end{equation}
We use the initial condition $(x_1(0), x_2(0))^T = (0,2)^T$, $\mu=100$, and $\omega=A=1$.
For our fourth example we first consider (see \cite{Fitzhugh1961, NagumoAY1964, KSS1997}) the one-dimensional Fitzhugh-Nagumo partial differential equation (PDE): 
\begin{equation}\label{eq:nagumopde}
\left\{ \begin{array}{lcr}
u_t = \phi(u) - v + \alpha \frac{\partial^2 u}{\partial x^2} \\
v_t = \varepsilon (u-\delta v) %-c(x,t)
\end{array} 
\right., \quad u=u(x,t), v=v(x,t) \in \mathbb{R}, \quad t> 0, x\in (0,1)
\end{equation}
with Neumann-type boundary conditions $u_x(0,t) = 0 = u_x(1,t)$ and $v_x(0,t) = v_x(1,t)$ and with $\phi$ given by $\phi(r) = -2r^3+6r$ and $\delta, \alpha,\varepsilon >0$.  We construct a system of ODEs by taking a uniform spatial discretization of \eqref{eq:nagumopde} with $x_j = j/J \equiv j \Delta x$ for $j=0,\hdots,J$ and the following finite difference approximation to $\Delta u(x_j,t)$ which takes into account the boundary conditions:
$$\Delta u(x_j,t) \approx D(u_j) := \left\{ \begin{array}{cc}
(u_{1}-u_{0})/(\Delta x)^2, & j=0 \\
(u_{J-1}-u_J)/(\Delta x)^2, &  j=J    \\
(u_{j+1}-u_{j-1} - 2u_j)/(\Delta x)^2, & \text{otherwise}
\end{array}
\right.$$
where $u_j(t) \approx u(x_j,t)$ for $j=0,\hdots,J$ and $\Delta x = 1/J$.  This leads to our fourth ODE which is the following $(2J+2)$-dimensional Fitzhugh-Nagumo system:
\begin{equation}\label{eq:nagumoode}
\left\{ \begin{array}{lcr}
\dot{u}_j = \phi(u_i) - v_i + \alpha D(u_j)\\
\dot{v}_j = \varepsilon(u_j - \delta v_j) %- c(x_j,t))
\end{array}
\right., \quad j=0,\hdots,J
\end{equation}
For parameter values we take $\varepsilon = 0.1$, $\alpha =0.3$, $\delta = 0.01$, and $J=14$ and for the initial condition we use $u_j(0) = \sin(0.5\pi j\Delta x)$ and $v_j(0) = \cos(0.5\pi j\Delta x)$ for $j=0,\hdots,J$.

\subsection{Nonautonomous stiffness detection}\label{sec:stiffdetect}

In this section we develop a method for stiffness detection based on approximating Steklov averages as defined in Equation \eqref{eq:steklovavg}.  Assume that \eqref{eq:lineq} satisfies Assumption \ref{as:assumption2}.  The conclusion of Lemma \ref{lem:localgrowthcts} implies that the Steklov averages \eqref{eq:steklovavg} of a corresponding upper triangular ODE measure average exponential growth/decay rates of solutions of \eqref{eq:lineq} on the interval $[t,t+\Delta t]$.  For a randomly chosen orthogonal $Q(t_0) = Q_0 \in \mathbb{R}^{d\times d}$ the Steklov averages of the corresponding upper triangular ODE $\dot{y} = B(t)y$ with initial orthogonal factor $Q_0$ tend to order themselves so that $s_1(t,\Delta t)$ corresponds to the right-most spectral interval and $s_d(t,\Delta t)$ corresponds to the left-most spectral interval \cite{DVV95}.  This motivates using the following as a stiffness indicator:
\begin{equation}\label{eq:QRstiffind}
S(t,\Delta t) = s_1(t,\Delta t)-s_d(t,\Delta t).
\end{equation}
If $S(t,\Delta t)$ is large in absolute value, then we expect that the problem is stiff and if $S(t,\Delta t)$ is near zero, then we expect that the problem is nonstiff.  We remark that in general $s_1(t,\Delta t) > s_d(t,\Delta t)$ holds on average, but does not hold point-wise; for sufficiently large $\Delta t$ the quantities $s_1(t,\Delta t)$ and $s_d(t,\Delta t)$ become approximations to respectively the right and left end-points of the Lyapunov and Sacker-Sell spectra.

We now discuss how to approximate $S(t,\Delta t)$ along a sequence of time-steps $h=\{h_n\}_{n=0}^{\infty}$.  Consider the numerical solution $x_{n+1} = \Phi^A(n;h)x_n$ of \eqref{eq:lineq} using a one-step method $\mathcal{M}$ with local truncation error of order $p \geq 1$.  We first approximate $s_1(t_n,h_n)$ as follows.  Given an initial $q_{0}\in \mathbb{R}^d$ with $\|q_0\|_2 = 1$ ( $\|\cdot\|_2$ is the Euclidean 2-norm) we inductively form $v_{n}:=\Phi^A(n;h)q_{n}$ followed by normalization: $q_{n+1} := v_{n}/\|v_{n}\|_2$ and $R^A_{1,1}(n;h) :=\|v_{n}\|_2$.  We approximate $s_1(t_n,h_n)$ by $\sigma_1(n):=\ln(R^A_{1,1}(n;h))/h_n$ which is justified since Theorem \ref{thm:maintheorem2} implies that $s_1(t_n,h_n) = \sigma_1(n) + \mathcal{O}(\|h\|_{\infty}^{p})$ for sufficiently small $\|h\|_{\infty}$.  We approximate $s_d(t_n,h_n)$ by applying the same method used to approximate $s_1(t_n,h_n)$ to the adjoint equation $\dot{x} = -A(t)^T x$.  This is justified since the left end-points of the Lyapunov and Sacker-Sell spectra of the adjoint equation are the right end-points of the Lyapunov and Sacker-Sell spectra of \eqref{eq:lineq}.  

Our approximation of $S(t,\Delta t)$ along a sequence of time-steps $\{h_n\}_{n=0}^{\infty}$ using window length $w \geq 0$ and $n \geq w$ is defined as as 
$$SI(n,w) =\frac{1}{t_{n+w+1}-t_{n-w}}\sum_{k=0}^{2w}(\sigma_1(n-w+k)-\sigma_d(n-w+k))/h_{n-w+k}.$$
For IVPs of nonlinear ODEs we compute $SI(n,w)$ by forming $A(t) := Df(x(t;x_0,t_0),t)$ at the approximate solution values.

For an implicit method, forming $\Phi^A(n;h)$ exactly requires solving a linear system of equations.  To avoid this we instead form $\tilde{\Phi}^A(n;h) = \Phi^A(n;h) + \mathcal{O}(\|h\|_{\infty}^{q})$ where $q = \text{min}\{p+1,3\}$ at each implicit time-step where $\tilde{\Phi}^A(n;h)$ is formed by applying the explicit 2nd order method HEU(2-2-1) to $\dot{x} = A(t)x$.  The structure of HEU(2-2-1) also allows us to avoid approximating $A(t)$ with stage values whose order of approximation is much lower than the order of the method.  

In addition to our Steklov average based method we implement the stiffness indicator, denoted as $\sigma[A(t)]$, that was introduced in Definition 4.1 of \cite{GSC2015} that is formulated in terms of the logarithmic norm of the Hermitian part $\text{He}[A(t)] : = (A(t)+A(t)^T)/2$ of $A(t)$. To simplify the computation of $\sigma[A(t)]$ we assume that we are using the Euclidean 2-norm.  As noted in \cite{GSC2015} this implies that $\sigma[A(t)]$ equals the smallest eigenvalue of $\text{He}[A(t)]$ subtracted from the largest eigenvalue of $\text{He}[A(t)]$.  

In general we cannot expect any relation between $SI(n,w)$ and $\sigma[A(t_n))]$ as exemplified in Figure \ref{fig:stiffind_2dlin}.  However, we can characterize when these two indicators should be close to equal.  Assume $w = 0$ and note that for any bounded and continuous $A(t)$ we have $X(t_{n+1},t_n) = e^{\int_{t_n}^{t_{n+1}}A(\tau)d\tau} + \mathcal{O}(\|h\|_{\infty}^2)$ for all sufficiently small $\|h\|_{\infty}$ where $X(t;t_n)$ is the solution of \eqref{eq:Atrans}.  Hence, if $X(t_{n+1};t_n)$ is well-conditioned for eigenvalue computations (such as when $A(t_n)$ is normal and $h_n$ is small), then the logarithms of the real parts of the eigenvalues of $\Phi^A(n;h)$ divided by $h_n$ should be approximately equal to the average of the eigenvalues of $\text{He}[A(t)]$ for $t \in [t_n,t_{n+1}]$.  Forming $\sigma_1(n)$ and $\sigma_d(n)$ is equivalent to performing one step of power iteration to approximate the real parts of the eigenvalues of $\Phi^A(n;h)$ and the associated adjoint coefficient matrix followed by taking logarithms and division by $h_n$.  If the largest (in terms of absolute value) eigenvalue of $He[A(t)]$ is significantly larger than the next, then power iteration converges rapidly implying that a single step of power iteration applied to $\text{He}[A(t_n)]$ should be approximately the logarithm of a single step of power iteration applied to $\Phi^A(n;h)$ divided by $h_n$.  The same statement holds for the adjoint coefficient matrix and the smallest eigenvalue of $He[A(t)]$.  It follows that $SI(n,w)$ and $\sigma[A(t_n))]$ should be close when $\|h\|_{\infty}$ is small, $w\approx 0$, $A(t)A(t)^T-A(t)^TA(t) \approx 0$, and the largest eigenvalues of $A(t_n)$ and $-A(t_n)^T$  dominate over the next largest.  

We now highlight the advantages of computing $SI(n,w)$ over $\sigma[A(t_n)]$.  We first note that approximating $\sigma[A(t_n)]$ is norm dependent \cite{GSC2015} while $SI(n,w)$ is not.  Accurately approximating $SI(n,w)$ depends on integral separation which is expected to be strong in a stiff IVP and does not require that $A(t_n)$ or $\text{He}[A(t_n)]$ be normal or well-conditioned for eigenvalue computations.  For small $w$, forming $SI(n,w)$ is less expensive than $\sigma[A(t_n)]$, since forming $SI(n,w)$ essentially requires only a single step of power iteration applied to $\Phi^A(n;h)$ and the associated adjoint coefficient matrix followed by taking logarithms and a linear combination of $2w$ terms, whereas forming $\sigma[A(t_n)]$ requires at least one step of power iteration or some other method for approximating eigenvalues.  This cost advantage is important in the next section where fast and accurate approximations to $\sigma_1(n)$ and $\sigma_d(n)$ are needed at each step.

We compare the performance of $SI(n,w)$ with $\sigma[A(t_n)]$ with the linear ODE \eqref{eq:2dlin}, the compost bomb equation, and the forced Van der Pol equation in Figures \ref{fig:stiffind_2dlin}, \ref{fig:stiffind_cpb9em2}, \ref{fig:stiffind_cpb3em1}, and \ref{fig:stiffind_vdp}.  Figures \ref{fig:stiffind_cpb9em2}, \ref{fig:stiffind_cpb3em1}, and \ref{fig:stiffind_vdp} show that $SI(n,w)$ and $\sigma[A(t_n)]$ produce qualitatively similar results when applied to the compost bomb and Van der Pol equations.  However, as evidenced in Figures \ref{fig:stiffind_2dlin}, \ref{fig:stiffind_cpb9em2}, and \ref{fig:stiffind_cpb3em1}, our approximation to $SI(n,w)$ is more sensitive to changes in the step-size even over intervals where the solution is nonstiff.  The 2D linear ODE \eqref{eq:2dlin} provides a clear example where the performance of $SI(n,w)$ is superior to that of $\sigma[A(t_n)]$, with $SI(n,w)$ detecting intervals over which the solver takes smaller and larger time-steps where $A(t)$ is respectively more or less non-normal, while $\sigma[A(t_n)]$ is approximately constant at all time-steps.  The window length $w$ of $SI(n,w)$ is chosen to produce smooth plots.  The values $|SI(n,w)|$ and $|\sigma[A(t_n)]|$ are plotted since it is absolute values rather than sign that indicate stiffness.

\begin{figure}

\centering
\includegraphics[width=\textwidth,height=8cm]{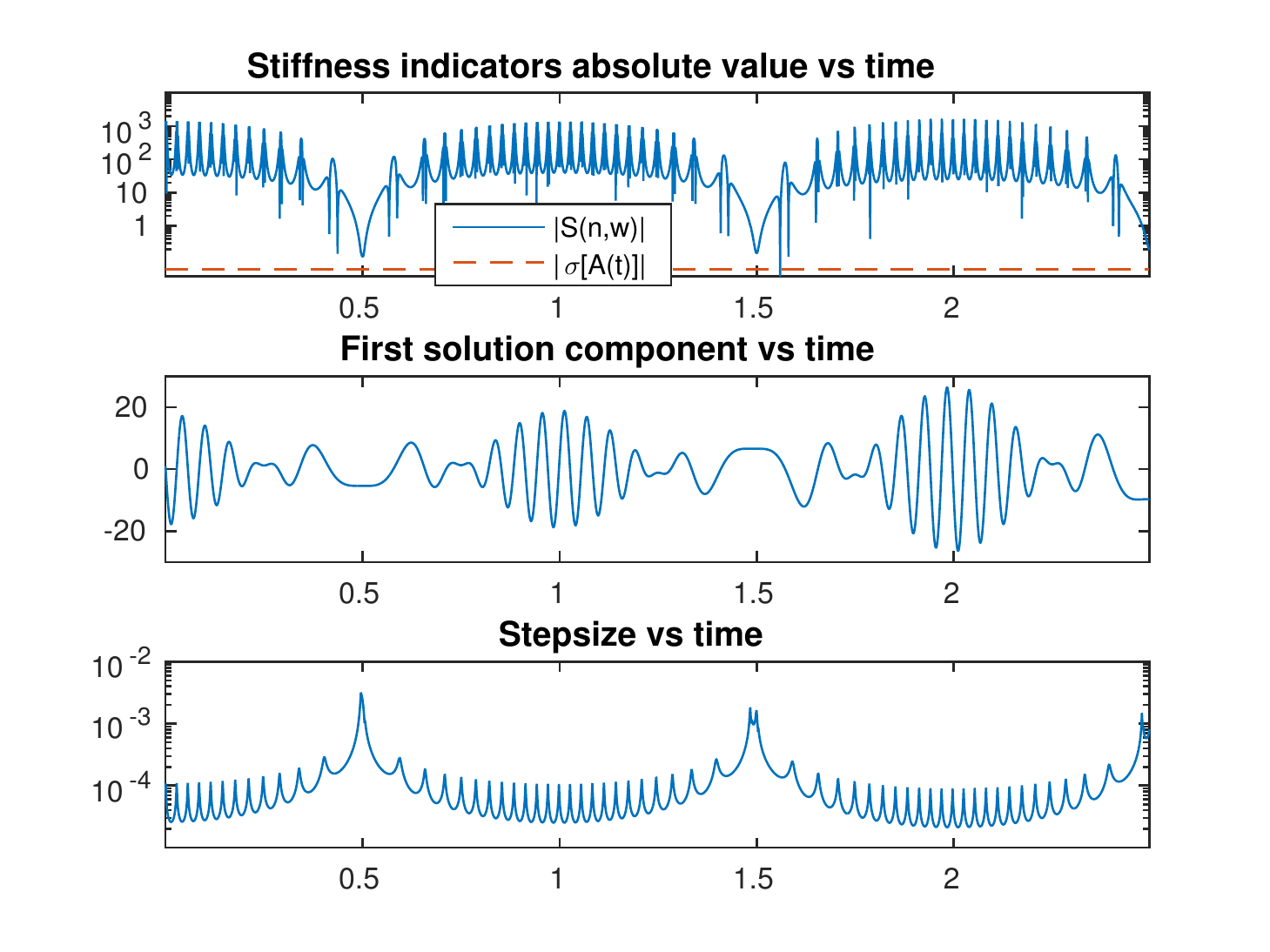}

\caption{Plots of approximate \ensuremath{|SI(n,w)|} with \ensuremath{w=1} and \ensuremath{|\sigma[A(t_n)]|}, first solution component, and step-size versus time for the numerical solution of the 2D linear ODE \protect\eqref{eq:2dlin} using the parameters specified in Section \ref{sec:testode} solved with HEU(2-1-2) using a relative and absolute error tolerance of \ensuremath{10^{-4}}.}
\label{fig:stiffind_2dlin}

\end{figure}

\begin{figure}
%\hspace{-1in}
\centering
\includegraphics[width=\textwidth,height=8cm]{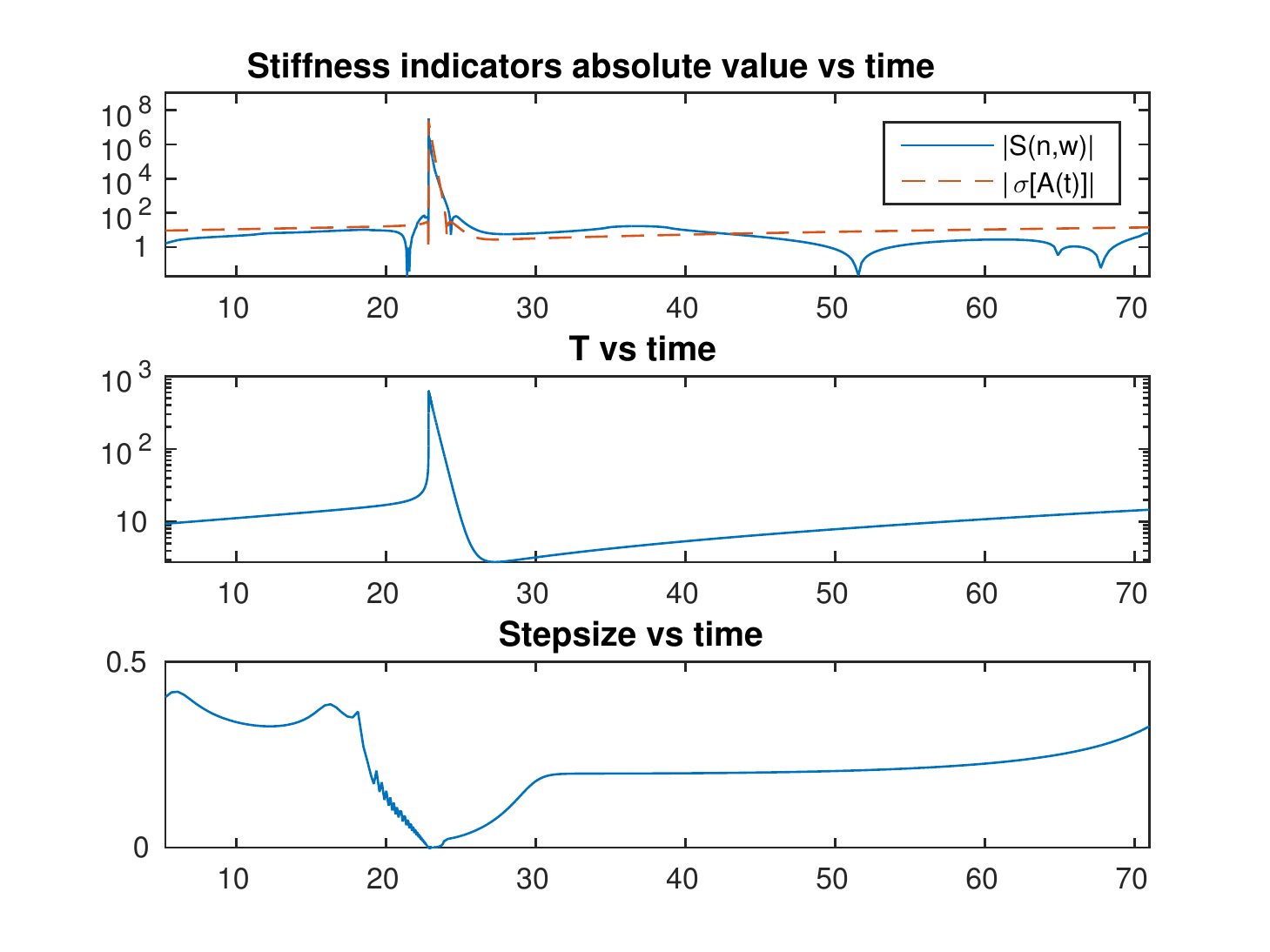}

\caption{Plots of approximate \ensuremath{|SI(n,w)|} with \ensuremath{w=20} and \ensuremath{|\sigma[A(t_n)]|}, first solution component \ensuremath{T}, and step-size versus time for the numerical solution of the compost bomb ODE \protect\eqref{eq:compbomb} using \ensuremath{\nu=0.09} solved with DP(7-5-4) using a relative and absolute error tolerance of \ensuremath{10^{-8}}.}
\label{fig:stiffind_cpb9em2}
\end{figure}

\begin{figure}
%\hspace{-1in}
\centering

\includegraphics[width=\textwidth,height=8cm]{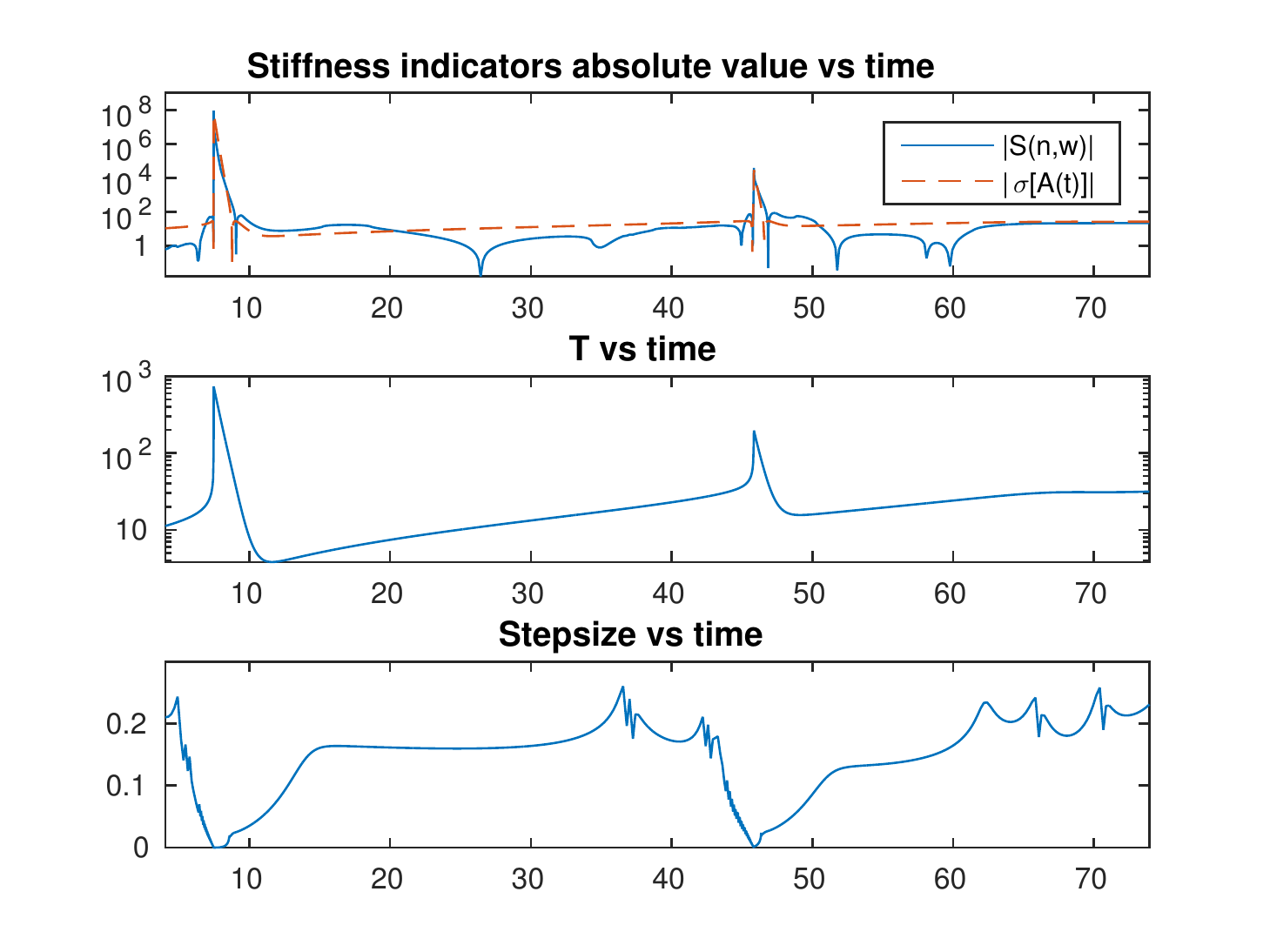}
\caption{Plots of approximate \ensuremath{|SI(n,w)|} with \ensuremath{w=20} and \ensuremath{|\sigma[A(t_n)]|}, first solution component \ensuremath{T}, and step-size versus time for the numerical solution of the compost bomb ODE \protect\eqref{eq:compbomb} using \ensuremath{\nu=0.30} solved with DP(7-5-4) using a relative and absolute error tolerance of \ensuremath{10^{-8}}.}
\label{fig:stiffind_cpb3em1}
\end{figure}

\begin{figure}
%\hspace{-1in}
\centering

\includegraphics[width=\textwidth,height=8cm]{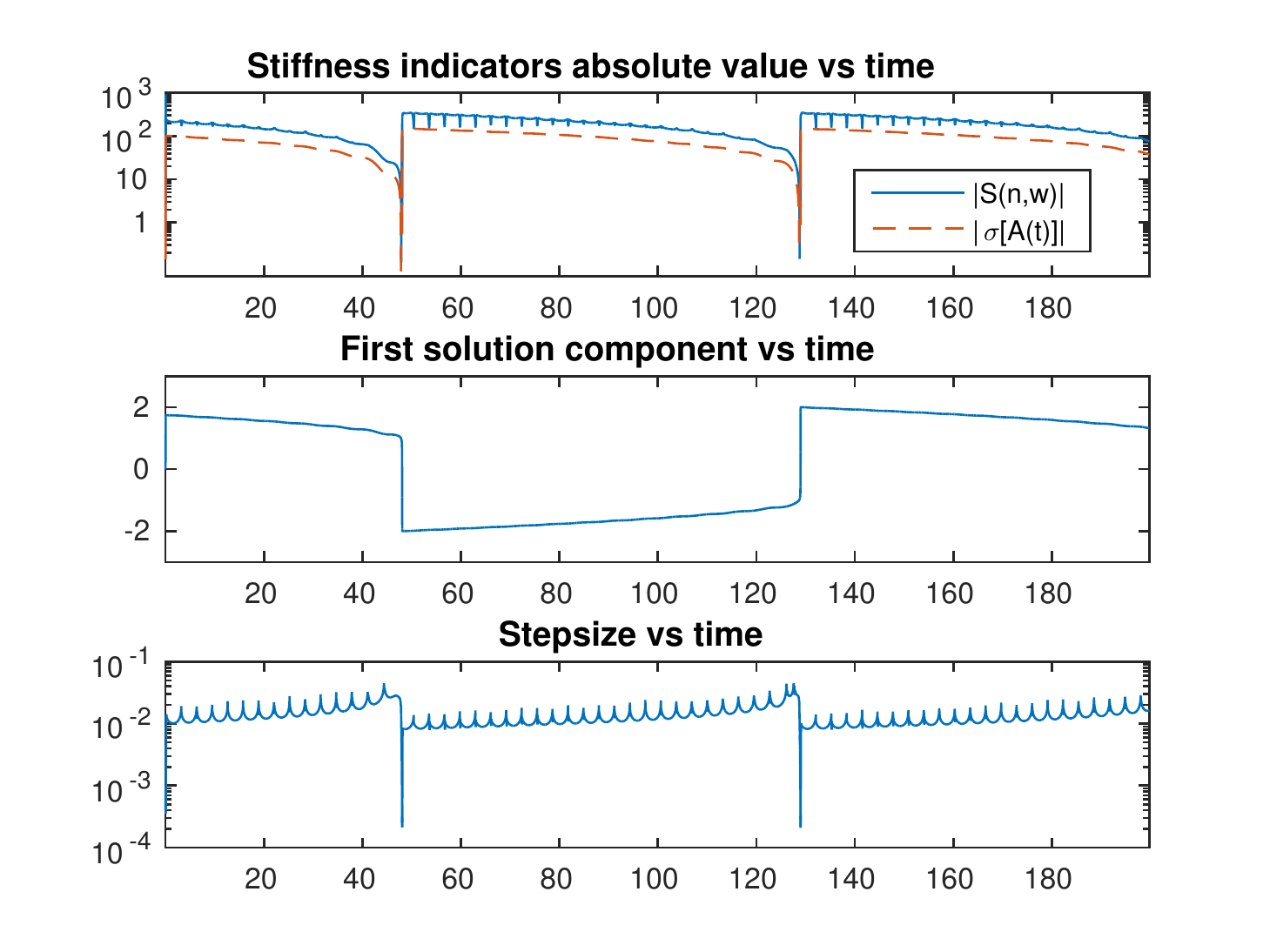}
\caption{Plots of approximate \ensuremath{|SI(n,w)|} with \ensuremath{w=10} and \ensuremath{|\sigma[A(t_n)]|}, first solution component, and step-size versus time for the numerical solution of the Van der Pol ODE \protect\eqref{eq:navdp2d} solved with DP(7-5-4) using a relative and absolute error tolerance of \ensuremath{10^{-8}}.}

\label{fig:stiffind_vdp}
\end{figure}

%\FloatBarrier

\subsection{QR implicit-explicit Runge-Kutta methods}\label{sec:qrimex}

 Consider an explicit Runge-Kutta method RKex($\nu$-$p$-$\hat{p}$) and an implicit Runge-Kutta method RKim($\hat{\nu}$-$p$-$\hat{p}$).  We construct a one-step method with local truncation error of order $p$, denoted as RKex($\nu$-$p$-$\hat{p}$)- RKim($\hat{\nu}$-$p$-$\hat{p}$)),  that switches between using the implicit and explicit Runge-Kutta methods as follows.  At time-step $t_n$ we form $\sigma_1(n)$ and $\sigma_2(n)$ as described in Section \ref{sec:stiffdetect}.  If $\sigma_1(n)$ is too small and negative or if $\sigma_2(n)$ is too large and positive, then we use the implicit scheme, otherwise we use the explicit scheme.  More precisely we use the explicit scheme if $\sigma_1(n) \leq d_1 H_{0} \text{ and } \sigma_2(n) \geq d_2 H_{0}$ where $d_1$ and $d_2$ are chosen according to the right and left endpoints of the linear stability domains of RKex($\nu$-$p$-$\hat{p}$) and RKim($\hat{\nu}$-$p$-$\hat{p}$) and the quantity $H_{0}$ is a parameter specifying the minimum allowable step-size restriction due to time-dependent stability that will be tolerated.  We refer to such implicit-explicit switching methods as QR-IMEX-RK methods and implement them with $odeqr$.

We approximate the parameter $H_{0}$ as follows.  Pick an interval over which the approximate solution is non-stiff.  Over this interval compute the mean step-size $h^m$ and set $H_0 = h^m \alpha$ where $\alpha > 0$ is a factor specifying the tolerance for how small the stability step-size restriction is relative to the mean non-stiff step-size. 

The compost bomb results (Table \ref{tab:imex_compbomb}) show that the QR-IMEX-RK method Mcpb3 has about the same mean step-size as the implicit method Mcpb2 for $\nu=0.09,0.30$ and $\text{TOL}=10^{-4},10^{-6},10^{-8}$ at a much lower cost in terms of function and Jacobian evaluations and linear solves.  The explicit method Mcpb1 must use a smaller time-step on average than Mcpb1 and Mcpb2 at all tested tolerances and at the tolerance of $10^{-4}$ the QR-IMEX-RK method Mcpb2 uses fewer function evaluations than the explicit method Mcpb1.  Parameter values $d_1$ and $d_2$ were chosen based on the stability regions of Mcpb1 and Mcpb2. 

We now discuss the results (Table \ref{tab:imex_nagpde}) of the discretized Fitzhugh-Nagumo PDE \eqref{eq:nagumoode}.  As the error tolerances are decreased the problem becomes stiffer leading to more uses of the implicit method by the QR-IMEX-RK methods.  The results in Table \ref{tab:imex_nagpde} show that at medium-low tolerances ($\text{TOL}=10^{-4},10^{-6}$) the explicit method Mfhn1 and the QR-IMEX-RK methods Mfhn2, Mfhn3, and Mfhn4 have about the same mean step-size and very few implicit steps are taken by the QR-IMEX-RK methods.  When tighter tolerances are used ($\text{TOL}=10^{-8},10^{-10}$) the problem is much stiffer and the QR-IMEX-RK solvers Mfhn2 and Mfhn3 whose implicit methods are L-stable are able to take larger time-steps on average than the explicit method Mfhn1 or the QR-IMEX-RK method Mfhn4 (which is not AN- or L-stable) at a cost of using more function and Jacobian calls than Mfhn1.  The implicit method of Mfhn4 is not AN- or L-stable and it never performs much better than the explicit method Mfhn1 in terms of average step-size or number of function evaluations.  This suggests that it is inadvisable to use an implicit method which is not AN- or L-stable as the implicit method in a QR-IMEX-RK method.  Notice that although Mfhn2 and Mfhn3 are able to take larger step-sizes on average than Mfhn1 for $\text{TOL}=10^{-8},10^{-10}$ the additional implicit time-steps cost more in terms of function and Jacobian evaluations, linear solves, and the overhead associated with forming $\sigma_1(n)$ and $\sigma_2(n)$ at each time-step.

\begin{table}[ht!]
\centering
\caption{Definitions of table entries.}
\label{tab:defs}
    \begin{tabular}{ | l| l |}
    \hline
    M              & Method used by the odeqr solver \\ \hline
    TOL               & Absolute and relative error tolerance (always taken to be equal)            \\ \hline
  %  $h_{\text{min}}$  & Minimum step-size \\ \hline
    $h_{\text{mean}}$ & Mean step-size \\ \hline
 %   $h_{\text{max}}$  & Maximum step-size\\ \hline
     nexp             & Number of explicit steps taken \\ \hline
     nimp             & Number of implicit steps taken \\ \hline
   %  nsuc             & Number of steps taken with no step-size reduction or increase\\ \hline 
 %    re           & Number of times the step-size was reduced or increased in a explicit step\\ \hline
 %    ri           & Number of times the step-size was reduced or increased in a implicit step\\ \hline
%     iter             & Mean number of iterations used by the Newton solver\\ \hline
     Feval            & Total number of evaluations by the ODE right-hand-side function\\ \hline
     Jaceval         & Total number of evaluations of the Jacobian $A(t)$\\ \hline
     Lsol         & Total number of linear solves\\ \hline
  %   IMEX             & Equals 1 if implicit step taken and 0 otherwise\\ \hline
      NA              & Not applicable \\ \hline

\hline     
\end{tabular}
\end{table}
\FloatBarrier
\hspace{2cm}
\begin{table}[ht!]
\caption{Experimental results for the explicit method Mcpb1 = HEU(2-2-1), the AN-stable implicit method Mcpb2 = SDIRK(2-2-1), and the IMEX-QR-RK method Mcpb3 = HEU(2-2-1)-SDIRK(2-2-1) solving the compost bomb ODE \protect\eqref{eq:compbomb} on the time interval $[0,80]$ for various values of TOL and $\nu = 0.09, 0.3$.  The QR-IMEX-RK parameters used were $d_1 = -2.0$ and $d_2 = 2.0$.  The quantity $H_0$ was computed by taking the mean step-size on the interval $[2,5]$ for $\nu =0.3$ and $[2,20]$ for $\nu = 0.09$ and using $\alpha =0.1$.  Jacobians are formed using finite difference approximations, the maximum step-size allowed was $0.5$, and the initial step-size is $h_0 = 0.05$.  See Table \ref{tab:defs} for definitions of the table entries.}
\label{tab:imex_compbomb}
\small
    \begin{tabular}{ | l |l |l |l | l | l | l | l | l | l | l | l | l}
\hline
$\nu = 0.09$ & & & & & & & &\\\hline
    M  & TOL & $H_0$ & $h_{\text{mean}}$  & nexp  & nimp  & Feval & Jaceval & Lsol\\ \hline
Mcpb1  & 1E-4 & 4.720E-3 & 2.369E-3 & 33751  & NA     & 135004 & NA     & NA\\ \hline
Mcpb2  & 1E-4 & 4.720E-3 & 6.179E-3 &  NA     & 12942 & 159880  & 134074  & 27044\\ \hline
Mcpb3  & 1E-4 & 4.720E-3 & 6.175E-3 & 3849   & 9103   & 124632  & 116964 & 18212 \\ \hline \hline

Mcpb1  & 1E-5 & 4.718E-4 & 1.351E-3  & 59234  & NA     & 236936 & NA     & NA\\ \hline
Mcpb2  & 1E-5 & 4.718E-4 &  1.954E-3  & NA     & 40943 & 494606 & 412840  & 82733\\ \hline
Mcpb3  & 1E-5 & 4.718E-4 & 1.954E-3  & 13123  & 27830  & 386436  & 360210  & 55660\\ \hline \hline

Mcpb1  & 1E-6 & 6.143E-4 &  5.627E-4 & 142173 & NA     & 568692  & NA & NA\\ \hline
Mcpb2  & 1E-6 & 6.143E-4 & 6.178E-4 &  NA     & 129492 & 1552910 & 1294114 & 258773\\ \hline
Mcpb3  & 1E-6 & 6.143E-4 & 6.177E-4 & 44658  & 84844   & 1196232 & 1106926 & 169558\\ \hline \hline

$\nu = 0.30$ & & & & & & & & \\ \hline
    M  & TOL & $H_0$ & $h_{\text{mean}}$  & nexp  & nimp  & Feval & Jaceval & Lsol\\ \hline
Mcpb1 & 1E-4 & 1.492E-3 & 1.005E-3  &  79611  & NA     & 318444 & NA     & NA\\ \hline
Mcpb2 & 1E-4 & 1.492E-3 & 3.954E-3  &  NA     & 20202  & 248622  & 208386  & 41988\\ \hline
Mcpb3 & 1E-4 & 1.492E-3 &  3.947E-3  &  7788   & 12447 & 180516  & 164980  & 24902\\ \hline \hline

Mcpb1 & 1E-5 & 1.943E-3 &  6.664E-4 &  120031 & NA    & 480124 & NA & NA \\ \hline
Mcpb2 & 1E-5 & 1.943E-3 &  1.251E-3 & NA     & 63942  & 769170 & 641516 & 128397\\ \hline
Mcpb3 & 1E-5 & 1.943E-3 &  1.251E-3 &  27048  & 36927 & 551288 & 497252 & 73859\\ \hline \hline

Mcpb1 & 1E-6 & 1.939E-3 & 3.196E-4 &  250297 & NA      & 1001188 & NA     & NA\\ \hline
Mcpb2 & 1E-6 & 1.939E-3 & 3.955E-4 &  NA     & 202264  & 2425590 & 2021310 & 404183\\ \hline
Mcpb3 & 1E-6 & 1.939E-3 & 3.954E-4 &  39943  & 162368  & 2106332 & 2026556 & 324924 \\ \hline

\hline     
\end{tabular}

\end{table}

\begin{table}[ht!]
\caption{Table of results for experiments on the spatially discretized Fitzhugh-Nagumo PDE \protect\eqref{eq:nagumoode} solved on the time interval $[0,200]$ using $J=14$.  The methods are the explicit method Mfhn1 = BS(4-2-3) and the QR-IMEX-RK methods Mfhn2 = BS(4-2-3)-ESDIRK(4-2-3), Mfhn3 = BS(4-2-3)-SDIRK(4-2-3) , and Mfhn4 = BS(4-2-3)-SDIRK(3-2-3).  The QR-IMEX-RK parameters were $d_1 = -3.5$ and $d_2 = 10.0$ and $H_0$ was computed by taking the mean step-size on the interval $[2,20]$ and using $\alpha = 0.5$.  Jacobians are formed exactly, the maximum step-size allowed is $0.5$, and the initial step-size was $h_0 = 0.05$.}
\label{tab:imex_nagpde}
\small
    \begin{tabular}{ | l  |l|l | l | l | l | l | l | l | l |}
    \hline
    M                 & TOL & $H_0$ &  $h_{\text{mean}}$ & nexp  & nimp & Feval & Jaceval & Lsol \\ \hline
Mfhn1 & 1E-4 & 5.383E-3 &  1.040E-2 & 9621 & NA     & 76968 & NA & NA\\ \hline 
Mfhn2 & 1E-4 & 5.383E-3 & 1.038E-2 & 9634 & 2     & 77128 & 19320 & 5\\ \hline 
Mfhn3 & 1E-4 & 5.383E-3 &  1.039E-2 & 9619 & 2   & 77008 & 19290 & 5\\ \hline 
Mfhn4 & 1E-4 & 5.383E-3 & 1.038E-2 & 9635 & 2     & 77128 & 19316 & 6\\ \hline   \hline

Mfhn1 & 1E-6 & 5.294E-3 &  1.013E-2 & 9873 & NA   & 78984 & NA & NA \\ \hline 
Mfhn2 & 1E-6 & 5.294E-3 & 1.027E-2 & 9353 & 383   & 85312 & 28428 & 928\\ \hline 
Mfhn3 & 1E-6 & 5.294E-3 &  1.031E-2 &  9358 & 346 & 84844 & 28028 & 904 \\ \hline 
Mfhn4 & 1E-6 & 5.294E-3 & 1.012E-2 &  9287 & 598  & 85690 & 29370 & 1301\\ \hline \hline

Mfhn1 & 1E-8 & 5.293E-3 & 7.946E-3 & 12586 & 0    & 100688 & NA & NA\\ \hline  
Mfhn2 & 1E-8 & 5.293E-3 & 8.730E-3 & 9067 & 2390  & 130296 & 71114 & 4830\\ \hline 
Mfhn3 & 1E-8 & 5.293E-3 & 8.927E-3 & 9158 & 2044  & 122776 & 63740 & 4145 \\ \hline 
Mfhn4 & 1E-8 & 5.293E-3 & 7.957E-3 & 9000 & 3568  & 136584 & 79016 & 7196\\ \hline \hline 

Mfhn1 & 1E-10 & 5.293E-3 & 3.689E-3 & 27109 & 0       & 216872 & NA & NA\\ \hline  
Mfhn2 & 1E-10 & 5.293E-3 & 4.768E-3 & 8352  & 12623   & 370560 & 295202 & 25345\\ \hline 
Mfhn3 & 1E-10 & 5.293E-3 & 5.183E-3 & 9365  & 11041   & 331424 & 259818 & 22133 \\ \hline 
Mfhn4 & 1E-10 & 5.293E-3 & 2.866E-3 & 0     & 34892   & 672350 & 637467 & 77167\\ \hline

\hline     
\end{tabular}

\end{table}

%=======================================================================================================================================
%=======================================================================================================================================
%=======================================================================================================================================
%=======================================================================================================================================
%=======================================================================================================================================
%=======================================================================================================================================

\FloatBarrier
\section{Afterword}\label{sec:afterword}

We have used QR approximation theory for Lyapunov and Sacker-Sell spectra to develop a time-dependent stability theory for one-step methods approximating time-dependent solutions to nonlinear and nonautonomous ODE IVPs.  This theory was used to justify characterizing the stability of a one-step method solving an ODE IVP with real-valued, scalar, nonautonomous linear test equations.  In the companion paper \cite{SVV2017dcds} we use invariant manifold theory for nonautonomous difference equations to prove the existence of an underlying one-step method for general linear methods solving time-dependent problems.  This is then used to extend our analysis of one-step methods to general linear methods.  It should also be possible to extend the theory developed in this paper to infinite dimensional IVPs (using the infinite dimensional QR approximation theory developed in \cite{BVV}) arising from PDEs where the step-size restriction will in general be replaced by a time-dependent CFL condition.

Detecting, quantifying, and understanding stiffness has the been a major research focus of the time discretization community for the past 60 years.  The methods we have developed in this work can be advantageous for problems with e,g. non-normal Jacobians where standard stiffness detection techniques, such as those using logarithmic norms or time-dependent eigenvalues, can potentially fail.  Our techniques are justifiable in terms of Lyapunov and Sacker-Sell spectral theory and at a low computational cost produce qualitatively the same information in situations where existing methods are effective and meaningful information where existing methods are ineffective.  In addition to providing quantitative information about stiffness, the QR and Steklov average based approach can provide other relevant information since the same quantities are used to estimate Lyapunov exponents and Sacker-Sell spectra.

%=======================================================================================================================================
%=======================================================================================================================================
%=======================================================================================================================================
%=======================================================================================================================================
%=======================================================================================================================================
%=======================================================================================================================================

\section{Appendix}\label{sec:appendix}

We present Butcher tableaux for the various Runge-Kutta methods used in Section \ref{sec:numexp} and state their stability and accuracy properties.  Methods are written in the form 
$$\begin{array} {c|c} 
c       & A \\ \hline
p       & b^T \\ 
\hat{p} & \hat{b}^T
\end{array} $$
where $A$ is the Runge-Kutta matrix with coefficient vectors $c$ and $b$.  The integer $p$ denotes the order of the method used to update the approximate solution and $\hat{p}$ denotes the order of the embedding method used to estimate the local error.

\begin{figure}[htb!]
$$\begin{array}{c|cc}
0 & 0   & \\
1 & 1   & 0 \\  \hline
1 & 1   &  0   \\
2 & 1/2 & 1/2  
\end{array} \quad 
\begin{array}{c|cc}
1 & 1   &      \\
0 & -1  & 1    \\  \hline
2 & 1/2 & 1/2 \\
1 & 1 & 0  
\end{array} $$
\caption{The Heun-Euler method HEU(2-2-1) and the SDIRK(2-2-1) methods.  The SDIRK(2-1-2) method is AN-stable with an A-stable embedding.}
\label{fig:heuneuler}
\end{figure}

\begin{figure}[htb!]
$$\begin{array}{c|ccccccc}
0    & 0                  &                     &                    &          &               &          & \\ \smallskip
1/5  &  1/5               &                     &                    &          &               &          & \\ \smallskip
3/10 &  3/40              & 9/40                &                    &          &               &          & \\ \smallskip
4/5  & 44/45              & -56/15              & 32/9               &          &               &          & \\ \smallskip
8/9  & \frac{19372}{6561} & -\frac{25360}{2187} & \frac{64448}{6561} & -212/729 &               &          &\\ \smallskip
1    & \frac{9017}{3168}  & -355/33     & \frac{46732}{5247} & 49/176   & -\frac{5103}{18656}   &          & \\ \smallskip
1    & 35/384     & 0                   & \frac{500}{1113}   & 125/192  & -\frac{2187}{6784}    & 11/84    & \\ \hline  \smallskip
5    & 35/384     & 0                   & \frac{500}{1113}   & 125/192  & -\frac{2187}{6784}    & 11/84    & \\ \smallskip
4    & \frac{5179}{57600} & 0                   & \frac{7571}{16695} & 393/640  & -\frac{92097}{339200} & \frac{187}{2100} & 1/40
\end{array} \quad \begin{array}{c|cccc}
0   & 0    &  0  & 0   & 0 \\
1/2 & 1/2  & 0   & 0   & 0\\
3/4 & 0    & 3/4 & 0   & 0 \\
1   & 2/9  & 1/3 & 4/9 & 0 \\ \hline
3   & 2/9  & 1/3 & 4/9 & 0  \\ 
2   & 7/24 & 1/4 & 1/3 & 1/8
\end{array} $$
\caption{The DP(7-5-4) method \cite{DP1980} (left) and the BS(4-2-3) method \cite{BS1989} (right).}
\label{fig:explicit}
\end{figure}

\begin{figure}[htb!]
$$\begin{array}{c|ccc}
5/6    & 5/6     &  0      & 0 \\
29/108 & -61/108 & 5/6     & 0 \\
1/6    & -23/108 & -33/61  & 5/6 \\ \hline
  3    & 25/61   & 36/61   & 0   \\
  2    & 26/61   & 324/671 & 1/11
\end{array} \quad \begin{array}{c|cccc}
1/4   & 1/4     & 0       & 0      & 0 \\
11/28 & 1/7     & 1/4     & 0      & 0 \\
1/3   & 61/144  & -49/144 & 1/4    & 0 \\
1     & 0       & 0       & 3/4    & 1/4 \\ \hline
3     & 0       & 0       & 3/4    & 1/4  \\
2     & -61/600 & 49/600  & 79/100 & 23/100
\end{array}$$
\caption{The SDIRK(3-2-3) method (left) \cite{NorsettThomsen1984} and SDIRK(4-2-3) method (right) \cite{CameronEtal2002}.  For SDIRK(3-2-3) both the 3rd order method and the 2nd order embedding are A-stable.  For SDIRK(4-2-3) the 3rd order method is A-stable and L-stable while the 2nd order embedding is A-stable. }
\label{fig:SDIRK}
\end{figure}

\begin{figure}[htb!]
$$ \begin{array}{c|cccc}

0  & 0     & 0       & 0      & 0 \\ \smallskip
\frac{1767732205903}{2027836641118} & \frac{1767732205903}{4055673282236}     & \frac{1767732205903}{4055673282236}      & 0      & 0 \\ \smallskip
\frac{3}{5}  & \frac{2746238789719}{10658868560708} & -\frac{640167445237}{6845629431997} & \frac{1767732205903}{4055673282236}       & 0 \\ \smallskip
1     & \frac{1471266399579}{7840856788654}  & -\frac{4482444167858}{7529755066697} & \frac{11266239266428}{11593286722821} & \frac{1767732205903}{4055673282236} \\ \hline \smallskip
3     & \frac{1471266399579}{7840856788654}  & -\frac{4482444167858}{7529755066697} & \frac{11266239266428}{11593286722821} & \frac{1767732205903}{4055673282236}  \\ \smallskip
2     & \frac{2756255671327}{12835298489170} & -\frac{10771552573575}{22201958757719}  & \frac{9247589265047}{10645013368117} & \frac{2193209047091}{5459859503100}

\end{array}$$
\caption{The ESDIRK(4-2-3) method \cite{CK2003} is A-stable and L-stable while the embedding is only A-stable.}
\label{fig:ESDIRK}
\end{figure}

\FloatBarrier



%\section{References}

\bibliographystyle{plainnat}      
\bibliography{bib_onestep}

\end{document}